\documentclass[11pt]{amsart}
\usepackage{amsthm,amsmath,amsxtra,amscd,amssymb,xypic,mathtools,accents}
\usepackage[colorlinks=true]{hyperref}
\usepackage[all]{xy}

\newcommand{\de}{\delta}
\newcommand{\ep}{\varepsilon}

\newcommand\chars[2]{\left[\begin{smallmatrix}#1\\ #2\end{smallmatrix}\right]}
\newcommand\tc[2]{\theta\chars{#1}{#2}}
\newcommand\tn{\theta_{\rm null}}
\newcommand\gn{\grad_{\rm null}}
\newcommand\Sn{\calS_{\rm null}}

\def\dim{\mathrm{dim}}

\def\ker{\mathrm{ker}}

\newcommand{\CC}{{\mathbb{C}}}

\newcommand{\HH}{{\mathbb{H}}}

\newcommand{\PP}{{\mathbb{P}}}
\newcommand{\QQ}{{\mathbb{Q}}}
\newcommand{\RR}{{\mathbb{R}}}
\newcommand{\ZZ}{{\mathbb{Z}}}

\def\MCG{\mathrm{MCG}}
\def\SY{\mathfrak{S}}

\def\calHM{\mathcal{HM}}

\newcommand{\T}{{\Theta}}

\newcommand{\calA}{{\mathcal A}}
\newcommand{\calC}{{\mathcal C}}

\newcommand{\calM}{{\mathcal M}}
\newcommand{\calJ}{{\mathcal J}}

\newcommand{\calX}{{\mathcal X}}
\newcommand{\calY}{{\mathcal Y}}

\newcommand{\calS}{{\mathcal S}}
\newcommand{\calH}{{\mathcal H}}

\newcommand{\op}{\operatorname}

\newcommand{\Sp}{\op{Sp}}

\def\calN{\mathcal{N}}
\newcommand{\Sing}{\op{Sing}}

\newcommand{\grad}{\op{grad}}

\newcommand\codim{\op{codim}}

\newcommand{\univ}[1]{\widehat{#1}}

\def\sm{\mathrm{sm}}
\def\sing{\mathrm{sing}}
\def\bm#1{\text{\boldmath$#1$}}

\renewcommand\tt[2]{\t\left[\begin{matrix}#1\\ #2\end{matrix}\right]}
\theoremstyle{plain}
\newtheorem{thm}{Theorem}[section]
\newtheorem{lm}[thm]{Lemma}
\newtheorem{prop}[thm]{Proposition}
\newtheorem{cor}[thm]{Corollary}
\newtheorem{fact}[thm]{Fact}

\newtheorem{notation}[thm]{Notation}

\newtheorem{maintheorem}{Theorem}

\newenvironment{mthm}[1]
{\renewcommand{\themaintheorem}{#1}
\begin{maintheorem}}
{\end{maintheorem}}

\newtheorem{maincorollary}{Corollary}

\newenvironment{mcor}[1]
{\renewcommand{\themaincorollary}{#1}
\begin{maincorollary}}
{\end{maincorollary}}

\newtheorem{mainproposition}{Proposition}

\newenvironment{mprop}[1]
{\renewcommand{\themainproposition}{#1}
\begin{mainproposition}}
{\end{mainproposition}}

\newtheorem{mainlm}{Lemma}

\theoremstyle{definition}
\newtheorem{dfx}[thm]{Definition}

\newtheorem{rem}[thm]{Remark}
\newtheorem*{nota}{Notation}

\newtheorem{example}[thm]{Example}

\def\CORRECT{}

\usepackage{stackengine}
\stackMath
\newcommand\tsup[2][2]{%
 \def\useanchorwidth{T}%
  \ifnum#1>1%
    \stackon[-.5pt]{\tsup[\numexpr#1-1\relax]{#2}}{\scriptscriptstyle\sim}%
  \else%
    \stackon[.5pt]{#2}{\scriptscriptstyle\sim}%
  \fi%
}

\usepackage{xcolor}

\def\wti#1{\widetilde{#1}}

\def\t{\vartheta}

\def\rar{\rightarrow}

\def\thra{\twoheadrightarrow}
\def\hra{\hookrightarrow}
\def\Pic{\mathrm{Pic}}

\def\lra{\longrightarrow}

\def\IRR{\mathrm{irr}}
\def\ol#1{\overline{#1}}
\def\pa{\partial}

\def\-{\text{-}}

\begin{document}

\title[Universally irreducible subvarieties of Siegel moduli spaces]{Universally irreducible subvarieties\\ of Siegel moduli spaces}

\author[G. Mondello]{Gabriele Mondello}
\address{Dipartimento di Matematica, Piazzale Aldo Moro, 2, I-00185 Roma, Italy}
\email{mondello@mat.uniroma1.it}

\author[R. Salvati Manni]{Riccardo Salvati Manni}
\address{Dipartimento di Matematica, Piazzale Aldo Moro, 2, I-00185 Roma, Italy}
\email{salvati@mat.uniroma1.it}
  
\begin{abstract}
A subvariety of a quasi-projective complex variety $X$
is called ``universally irreducible'' if its preimage inside
the universal cover of $X$ is irreducible.
%
%Given a quasi-projective complex variety $X$ with universal cover $\wti{X}$,
%a subvariety of $X$ whose preimage in $\wti{X}$ is irreducible  is called %``universally irreducible''.
In this paper we investigate sufficient conditions for universal irreducibility.
We consider in detail \CORRECT{complete intersection subvarieties} of small codimension
inside Siegel moduli spaces of any finite level.
Moreover we show
that, for $g\geq 3$, every Siegel modular form is the product of finitely many
irreducible analytic functions on the Siegel upper half-space $\HH_g$.
%
%In this setting we show that that absolutely irreducible subvarieties of $\calA'_g$, i.e.~subvarieties that  are irreducible in modular varieties $\calA''_g$ of finite level that dominate $\calA'_g$, of codimension at most $g-2$, are universally irreducible.
We also discuss the special case of singular theta series of weight $\frac{1}{2}$ and of Schottky forms.
%In particular the discussion will focalize around a problem involving 
% In this  paper using basic results about local irreducibility of analytic spaces and their universal coverings we get several results.  
\end{abstract}

%\date{\today}
\maketitle
%\tableofcontents

\section{Introduction}

\subsection{Motivation}\label{sec:main}

The {\it{moduli space $\calA_g$ of complex principally polarized Abelian varieties of dimension $g$}} can be obtained as the quotient of the {\it{Siegel upper half-space}}
\[
\HH_g:=\left\{
\tau\in\mathrm{Mat}_{g\times g}(\CC) 
\ \text{symmetric, with $\Im(\tau)>0$}\right\}
\]
by $\Gamma_g:=\Sp_{2g}(\ZZ)$ that acts with finite stabilizers.
Hence, in the orbifold sense, 
the projection $\HH_g\rar\calA_g$ can be seen as a
universal covering space
and
every finite-index subgroup $\Gamma$ of $\Gamma_g$ determines
a finite \'etale cover $\calA_g(\Gamma):=\HH_g/\Gamma$ of $\calA_g$.
Since holomorphic line bundles on $\HH_g$ are trivializable,
sections of holomorphic line bundles (in the orbifold sense) on $\calA_g(\Gamma)$
lift to {\it{Siegel modular forms with respect to $\Gamma$}}, namely
holomorphic functions $F:\HH_g\rar\CC$ that enjoy a suitable
equivariance property with respect to the action of $\Gamma$ on $\HH_g$
(see Section \ref{ssc:siegel}).

The zero locus of a modular form with respect to $\Gamma$
determines a $\Gamma$-invariant divisor in $\HH_g$, which is the pull-back
of an effective divisor on $\calA_g(\Gamma)$.
Vice versa, for $g\geq 3$, every effective divisor in any $\calA_g(\Gamma)$ pulls back
to $\HH_g$ to the zero locus of a modular form.
%of semi-integral weight, possibly with a multiplier

A basic example of modular forms are the {\it{theta constants
with semi-integral characteristic $\chars\ep\de$}}
defined as
\[
\theta\chars\ep\de(\tau):=\sum_{n\in\ZZ^g}
\exp\,\pi i\left[ \left(n+\ep/2\right)^t \tau\left(n+\ep/2\right)+\left(n+\ep/2\right)^t\de\right].
\]
where $\ep,\de\in(\ZZ/2)^g$.
Several years ago E.~Freitag and the second named author discussed at length about the problem of factorizing such modular forms.

In \cite[Theorem 4.6]{FR}, Freitag proved that these modular forms are {\it{absolutely irreducible}}, i.e.~they are nonzero and their divisors are irreducible in $\calA_g(\Gamma)$ with respect to arbitrary small finite-index subgroups $\Gamma$ of  $\Gamma_g$. We call {\it{universally irreducible}}
a non-zero modular form that has a stronger property, namely of being
irreducible as an analytic function on $\HH_g$.

%
%One advantage of viewing $\calA_g(\Gamma)$ as a quotient of $\HH_g$ 
%is to work, at least analytically, on the 
%explicit open subset $\HH_g$ of $\CC^{g(g+1)/2}$, which is also isomorphic to
%a bounded homogeneous domain.  
%

%the problem of considering the preimage of these divisors and other  absolutely prime forms on the universal  covering $\HH_g$
%, i.e.~the Siegel upper half-space, 
%is rather  natural. 
%{\bf  TOGLIERE Another explicit example of modular form of higher weight is the so-called Schottky form $f_g$ with $g\geq 4$.

%\begin{qu}[Irreducible factorization of Schottky forms]
%Determine the factorization of 
%\[
%f_g(\tau):=\theta_{E_8\oplus E_8}(\tau)-\theta_{D_{16}^+}(\tau)
%\]
%into irreducible analytic functions, where
%\[
%\theta_S (\tau):=\sum_{G\in \mathrm{Mat}_{m,g}(\ZZ)}\mathrm{exp}(\pi i\cdot\mathrm{tr}( G^{t}SG\tau  ))
%\]
%for an integral, positive-definite quadratic form $S$ on $\ZZ^m$.
%\end{qu}  FINO A QUI}
%{\bf  Tolto Schottky}\\

Such irreducibility properties have a geometric
counterpart.

\begin{dfx}[Absolute and universal irreducible subvarieties]
A subvariety $Z$ of $\calA_g(\Gamma)$
is {\it{absolutely irreducible}} if the preimage of $Z$
in every finite \'etale cover (in the orbifold sense) of $\calA_g(\Gamma)$ is irreducible,
and it is {\it{universally irreducible}} if the preimage of $Z$ in $\HH_g$
is irreducible.
\end{dfx}
%Thus, the divisor associated to an absolutely prime (resp. universally prime)
%form is absolutely irreducible (resp. universally irreducible).
 % In this  paper using basic results about local irreducibility of analytic spaces and their universal coverings we get several results.  

In this paper we investigate sufficient conditions for universal irreducibility of subvarieties of $\calA_g(\Gamma)$.

\subsection{Setting and conventions}

% We need  to introduce some  notations and  a consideration.
%  
%   We  denote  by  $\Gamma_g:=\Sp_{2g}(\ZZ)$  the {\it{Siegel modular group}}, and let 

For every $n\geq 2$ we let
\[
\Gamma_g(n):=\{\sigma\in\Gamma_g\ |\ \sigma\equiv I_{2g}\pmod{n}\}
\] 
denote the 
{\it{$n$-th principal congruence subgroup}} of $\Gamma_g$
and we briefly denote $\calA_g(\Gamma_g(n))$ by $\calA_g(n)$.

Now fix $n\geq 3$. The moduli space $\calA_g(n)$ is a smooth quasi-projective variety and 
$\HH_g\rar\calA_g(n)$ is its universal cover in the standard sense;
moreover, the same holds for every finite-index subgroup of $\Gamma_g(n)$.

%
%
% The quotient $\calA_g=\HH_g/\Gamma_g$ is the moduli space of complex principally polarized abelian varieties (ppav), and $\calA_g(n)=\HH_g/\Gamma_g(n)$ is the moduli space of ppav with a choice of a full symplectic level~$n$ structure.
%
%When viewed as maps of orbifolds,
%the natural projections $\HH_g\rar\calA_g$
%are unbranched covers and $\HH_g$ identifies to the universal cover of $\calA_g$.
%We recall incidentally that, 
%
%
%We denote by $p:\HH_g\to\calA_g$ and $p_n:\HH_g\to\calA_g(n)$ the quotient maps,
%so that $\HH_g$ identifies to the universal cover of $\calA_g$
%(in the orbifold sense). We recall incidentally that, for $n\geq 3$,
%the moduli space $\calA_g(n)$ is a smooth variety and $p_n$ is its universal cover in the standard sense.
%  
  
Let now $\Gamma$ be any finite-index subgroup of $\Gamma_g$.
Then $\calA_g(\Gamma_g(n)\cap\Gamma)$ is a
smooth quasi-projective variety.
Thus $\calA_g(\Gamma)$ is the quotient of $\calA_g(\Gamma_g(n)\cap\Gamma)$
by the finite group $\Gamma/(\Gamma_g(n)\cap\Gamma)$.
Hence $\calA_g(\Gamma)$ always has the structure of smooth complex-analytic orbifold and of smooth Deligne-Mumford stack.

In order to have a more uniform treatment for all finite-index subgroups of $\Gamma_g$,
we will always take the orbifold point of view. This means
that the words {\it{smooth}}, {\it{singular}}, {\it{\'etale cover}}, {\it{fundamental group}} must be understood in the orbifold sense. 

%Note that in Sections \ref{sec:stratification}-\ref{sec:cpx-irred}-\ref{sec:LHT}-\ref{sec:connected}-\ref{sec:lift}-\ref{sec:absol} we work with complex-analytic spaces, though.

\subsection{Main results}\label{ssc:main}

Let $\Gamma$ be any finite-index subgroup of $\Gamma_g$.
%and let $\calA_g(\Gamma):=\HH_g/\Gamma$.

Our first main result is the following.
 
\begin{mthm}{A}\label{mthm:div}
Let $g\geq 3$ and let $D\subset \calA_g(\Gamma)$ be an effective divisor.
Then the preimage $\wti{D}$ in $\HH_g$ of $D$ is connected.
Moreover the following hold.
\begin{itemize}
\item[(i)]
If $D$ is locally irreducible, then $D$ is universally irreducible.
\item[(ii)]
If $D$ is absolutely irreducible, then $D$ is universally irreducible.
\end{itemize}
\end{mthm}

The connectedness claim (proven in Theorem \ref{thm:connectedness})
relies on a generalization (Theorem \ref{thm:divisor}) of the homotopical Lefschetz hyperplane section theorem for the fundamental group  of  smooth quasi-projective varieties that have a projective model with small boundary. 
We mention that, for $g=2$, connectedness of $\wti{D}$ holds for 
divisors $D$ whose closures intersect the boundary of the Satake compactification  of $\calA_g(\Gamma)$ in a finite set of points, e.g.~divisors defined by Eisenstein series.

Assertion (i) (proven in
Theorem \ref{thm:connectedness} too)
is a direct consequence of the connectedness of $\wti{D}$. 
Here we recall that normality
implies local irreducibility (see Lemma \ref{lm:irred}(iii)) and so
claim (i) also applies to normal divisors.

Finally, assertion (ii) is a special case of the following.
 
\begin{mthm}{B}\label{mthm:codim}
Let $g\geq 3$ and let $Z\subset\calA_g(\Gamma)$ be an absolutely irreducible 
\CORRECT{complete intersection subvariety} of codimension at most $g-2$.
Then $Z$ is universally irreducible.
\end{mthm}

Actually, the above Theorem \ref{mthm:codim} is a consequence of the following (see Theorem \ref{thm:ag} in the body of the paper).

\begin{mthm}{B'}\label{mthm:codim-B'}
Let $g\geq 3$ and let $Z\subset\calA_g(\Gamma)$ be any \CORRECT{complete intersection subvariety} of codimension at most $g-2$.
Then its preimage $\wti{Z}$ in $\HH_g$ has finitely many irreducible components.
\end{mthm}

 As noted by Freitag \cite[Theorem 4.7]{FR}, for $g\geq 3$
 all modular forms can be factorized 
 into absolutely irreducible ones (Lemma \ref{lm:factorization}).
Such result heavily relies on the fact that
$\mathrm{Pic}(\HH_g/\Gamma)\otimes\QQ= \QQ\cdot \lambda$
 and  the divisibility of the integral class $\lambda$  
in $\mathrm{Pic}(\HH_g/\Gamma)/\!\mathrm{tors}$
is uniformly bounded from above for every finite-index
subgroup $\Gamma$ of $\Gamma_g$
(see Section \ref{ssc:topological-Ag}).

Combining Freitag's observation with Theorem \ref{mthm:div}, we immediately have the following.

\begin{mcor}{C}\label{mcor:prime-forms}
For $g\geq 3$, every modular form is a finite product
of universally irreducible ones.
\end{mcor}

  \noindent

Even with Theorem \ref{mthm:codim} in our hands,
it is not always easy to verify whether a subvariety $Z$ of $\calA_g(\Gamma)$
is absolutely irreducible.
Thus, in our last main result, we provide a 
criterion for connectedness and irreducibility
of the preimage in $\HH_g$ of $Z$, which sometimes turns useful.
%We stress that such method applies to subvarieties that are not necessarily divisors.
 
In order to state it, we recall that the Jacobian locus (resp.~the hyperelliptic locus)
in $\HH_g$ is the locally closed locus
of period matrices corresponding to Jacobians of smooth curves 
(resp.~of smooth hyperelliptic curves) of genus $g$.
The Jacobian and the hyperelliptic loci in $\calA_g(\Gamma)$ are the image of the 
Jacobian and hyperelliptic loci of $\HH_g$
via the projection $\HH_g\rar\calA_g$.

\begin{mprop}{D}\label{mprop:sub}
%[Subvarieties containing the Jacobian locus]
Let $Z\subset\calA_g(\Gamma)$ be an irreducible subvariety
and let $\wti{Z}$ be its preimage in $\HH_g$.
\begin{itemize}
\item[(i)]
If a Zariski-open subset $Y$ of the Jacobian locus 
is contained in $Z$, then $\wti{Z}$ is connected, with finitely many irreducible components.
Moreover, if $Y$ is not entirely contained inside the singular locus of $Z$, then $\wti{Z}$ is irreducible.
\item[(ii)]
Assume $\Gamma\subseteq\Gamma_g(2)$.
If a Zariski-open subset $Y$ of 
the hyperelliptic locus 
is contained in $Z$, then $\wti{Z}$ is connected, with finitely many irreducible components.
Moreover, if $Y$ is not entirely contained inside the singular locus of $Z$, then $\wti{Z}$ is irreducible.
\end{itemize}
\end{mprop}

The above result (Proposition \ref{prop:irr-Mg} in the body of the article)
is essentially a consequence of the topological considerations recalled
in Section \ref{sec:topological} together with
the surjectivity of the symplectic representations of the mapping class group (Fact \ref{fact:rho-surjective}) and 
of the hyperelliptic mapping class group at level 2 (Proposition \ref{prop:pi1-hyp}).

%
%\begin{itemize}
%\item
%$\wti{Z}$ is connected if and only if $\pi_1(Z)\rar\Gamma$ is surjective
%(follows from Lemma \ref{lm:lift}(iii))
%\item
%$\wti{Z}$ is irreducible if and only if $\pi_1(Z_{\sm})\rar\Gamma$ is surjective (Corollary \ref{cor:criterion}(i))
%\item
%the fundamental group of the Jacobian locus of $\calA_g(\Gamma)$ surjects onto $\Gamma$
%(see, for instance, \cite[Theorem 6.4]{farb-margalit})
%\item
%the fundamental group of a hyperelliptic component 
%(namely, a connected component of the hyperelliptic locus)
%in $\calA_g(\Gamma)$ surjects
%onto $\Gamma$, if $\Gamma\subset\Gamma_g(2)$
%(see \cite[Theorem 1]{ac} or \cite[Lemma 8.12]{mutheta2}).
%\end{itemize}

As applications of the above technique,
we show that even theta-nulls are universally irreducible for $g\geq 3$
(Corollary \ref{cor:tn}),
that each component of the moduli space of intermediate Jacobians of cubic threefolds in $\calA_5 (2)$ is universally irreducible
(Corollary \ref{cor:grad}),
and that the Schottky form is universally irreducible for $g\geq 4$
(Corollary \ref{cor:sch}).
Finally, we also describe how different
the situation with even theta-nulls is in genus two (Proposition \ref{prop:th2}).

%The proof of Theorem \ref{mthm:codim} (Theorem \ref{thm:ag} in the body of the article)
%relies on the following statements:
%\begin{itemize}
%\item 
%absolute and universal irreducibility are invariant under the action by elements of the rational  symplectic  group
%(Lemma  \ref{lemma:transl-irr})
%\item
%each subvariety of codimension at most $g-2$ has a rational translate that meets the Jacobian locus in a non-hyperelliptic Jacobian (Lemma \ref{lemma:degenerations})
% \item
% let $\ol{\calM}_g$ be the Deligne-Mumford compactification of the moduli space of curves
%of genus $g$
%and $\delta^{\IRR}_h$ be the locally-closed boundary stratum in $\ol{\calM}_g$
%that parametrizes curves whose normalization has genus $g-h$;
%then, for any irreducible subvariety $Y\subset\calM_g$
%whose closure in $\ol{\calM}_g$
%meets $\delta^{\IRR}_h$ in a non-compact subset,
%the image of $\pi_1(Y_{\sm})\rar\Gamma_g$ contains
%two commuting transvections associated to a pair of linearly independent vectors of $\ZZ^{2g}$ (Proposition \ref{prop:two-twists})
%\item
%Zariski-dense subgroups of $\Gamma_g$ that contain commuting transvections
%associated to a pair of linearly independent vectors of $\ZZ^{2g}$
%have finite index (Proposition \ref{prop:commuting}, based on \cite[Theorem 1.2]{SV}).
%\end{itemize}
%

\subsection{Structure of the paper}
Besides the present introduction, the paper has seven more sections and one appendix.

In Section \ref{sec:topological} we collect some standard
facts about the topology of complex-analytic spaces 
(Sections \ref{sec:stratification}-\ref{sec:cpx-irred})
and we prove some criteria for connectedness
and irreducibility of liftings (Sections \ref{sec:connected}-\ref{sec:lift}).

In Section \ref{sec:LHT} we prove the version of the Lefschetz hyperplane theorem mentioned in Section \ref{ssc:main}.

In Section \ref{sec:absol} we distill some topological properties
needed in a more general setting to have
the existence of a factorization into absolutely irreducible divisors
and to rephrase Theorem \ref{mthm:codim} and Theorem \ref{mthm:codim-B'}.
%(Proposition \ref{prop:prime}).

In Section \ref{sec:Ag} we recall some topological properties of 
$\calA_g$ or of its finite covers $\calA_g(\Gamma)$.
Using the tools developed in Section \ref{sec:topological},
we prove the connectedness claim and part (i) of Theorem \ref{mthm:div} and
Proposition \ref{mprop:sub}.
Moreover, we discuss universal irreducibility of subvarieties of $\calA_g(\Gamma)$
that contain the Jacobian or the hyperelliptic locus.
In particular we analyze the case of even theta constants.

In Section \ref{sec:universal}
we prove Theorem \ref{thm:ag} (which is a simultaneous formulation of
Theorem \ref{mthm:codim} and Theorem \ref{mthm:codim-B'}),
using the results contained in Section \ref{sec:topological} and Section \ref{app:arithm}, and
in Appendix \ref{sec:dehn}.
In particular, we review the construction of a rational translate of a subvariety of $\calA_g(\Gamma)$.
%needed in the proof of Theorem \ref{mthm:codim}.

In Section \ref{sec:examples} we discuss
three examples: the zero locus of even theta constants in genus $2$,
the locus of intermediate Jacobians of cubic threefolds (in genus $5$)
and the Schottky form.

In Section \ref{app:arithm} we prove an arithmeticity criterion
for subgroups of $\Gamma_g$ needed in the proof of Theorem \ref{thm:ag}.

In Appendix \ref{sec:dehn}
we prove some easy facts about the symplectic monodromy at infinity
of subvarieties of the moduli space of curves.

 \subsection{Acknowledgements}  
 We are indebted to Eberhard Freitag   for  stimulating discussions,   suggestions and many useful comments on an earlier version of the manuscript.  We thank Mark Goresky   for useful discussions  on  Lefschetz hyperplane section theorem and  Julia Bernatska  for   illustrating her result on gradients of theta functions. We are also grateful to an anonymous referee for carefully reading the paper and for valuable suggestions.\\
 
The first-named author was partially supported by GNSAGA research group. Both authors were partially supported by PRIN 2022 grant ``Moduli and Lie theory''.

 %%%%%%%%%%%%%%%%%%%%%%%%%%%%%%%%%%%%%%%%%%%%%%%%%%%%%%%%%%%%%%%%%%%%%%%%%%
 %%%%%%%%%%%%%%%%%%%%%%%%%%%%%%%%%%%%%%%%%%%%%%%%%%%%%%%%%%%%%%%%%%%%%%%%%

 \section{Covering spaces, connectedness and irreducibility}\label{sec:topological}

In the following section we collect
sufficient conditions that ensure that, via a covering map of a smooth variety $X$,
the preimage of an irreducible subvariety $Z\subset X$ is connected or irreducible.

For connectedness we have to estimate the fundamental group
of the subvariety $Z$ and, in particular, the image of $\pi_1(Z)\rar\pi_1(X)$. For irreducibility
we have to deal with the fundamental group of
the smooth locus of $Z$, which can be more subtle if $Z$ is not normal
and in particular is not locally irreducible.

We introduce a simple technique that allows us to gain some control
on such fundamental groups: it consists in finding
an irreducible, locally irreducible
subvariety $Y\subset Z$ for which we have better understanding of the image of $\pi_1(Y)\rar\pi_1(X)$.
%The second technique can be employed when $Z$ is a hyperplane section of $X$
%(or a complete intersection of such) and $X$ satisfies suitable properties.

We state the results we are interested in for complex-analytic spaces. We then explain
the needed modifications in the case of orbispaces.

%; see also Section \ref{sec:orbispaces}.

%{\bf{REFERENCES???}}

\subsection{Complex varieties and links}\label{sec:stratification}

Here we collect some classical and basic facts about the topology
of complex analytic spaces. 
Unless differently specified, we work with the classical topology.\\

%\begin{prop}[Simultaneous triangulability of analytic spaces]
%\subsubsection{Stratifications and links}
%

Since part of the results we will recall involve
the topology of neighbourhoods of (possibly singular) subvarieties
of (possibly singular) analytic spaces, we consider stratifications
whose locally closed strata are locally as simple as possible.

\begin{dfx}[Locally trivial analytic subspaces]
Let $X$ be a reduced analytic space.
An analytic subspace $Y\subseteq X$ is {\it{locally trivial}}
if there exists a neighbourhood $U_{X/Y}$ of $Y$ and a projection $U_{X/Y}\rar Y$
such that
\begin{itemize}
\item
the fiber $(U_{X/Y})_y$ over $y\in Y$ is the cone over the {\it{link $(L_{Y/X})_y:=\pa U_y$ of $Y$ inside $X$ at the point $y$}}
\item
$U_{X/Y}\rar Y$ and $L_{X/Y}:=\bigcup_{y\in Y}(L_{X/Y})_y\rar Y$
are topologically locally trivial fiber bundles.
\end{itemize}
\end{dfx}

Let $X$ be a reduced analytic space, and
let $\{X_i\}_{i\in I}$ be a {\it{stratification of $X$}},
namely $I$ is a partially ordered set,
$\bigsqcup_i X_i=X$ and $\ol{X}_i=\bigcup_{j\leq i}X_j$, where each $\ol{X}_i$ is an analytic subspace of $X$
and $X_i$ is Zariski open inside $\ol{X}_i$.

\begin{dfx}[Good stratifications]\label{dfx:good-strat}
A stratification $\{X_i\}_{i\in I}$ of $X$ is {\it{good}} if
\begin{itemize}
\item
each locally closed stratum $X_i$ is connected and non-singular
\item
$X$ admits a locally finite
triangulation such that every open simplex is contained
in a unique $X_i$
\item
every $X_i$ is locally trivial inside $X$.
\end{itemize}
\end{dfx}

In the setting of the above Definition \ref{dfx:good-strat},
we call $U_i:=U_{X_i/X}$ 
the {\it{tubular neighbourhood of $X_i$}} and 
$L_i:=L_{X_i/X}$ the {\it{link of $X_i$}} inside $X$,
and we denote by $\dot{U}_i$ the complement of $X_i$ inside $U_i$, which is a locally trivial $(0,1]$-bundle over $L_i$.
Note that the tubular neighbourhood and the link of $X_i$ inside $X$
admit the natural stratifications 
$U_{X_i/X}=\bigsqcup_{j\geq i} (U_{X_i/X}\cap X_j)$
and $L_{X_i/X}=\bigsqcup_{j\geq i} (L_{X_i/X}\cap X_j)$.

%Note that, if $X$ admits a good stratification, then $X$ is locally contractible
%and so connectedness is equivalent to path-connectedness.

It is well-known that complex-analytic spaces admit a Whitney stratification and that Whitney stratifications are good in the above sense. Also the above-mentioned
stratifications of links and of tubular neighbourhoods of locally closed strata of a Whitney stratification are themselves Whitney stratifications.
We refer to \cite[Part I, section 1.2]{GM} for the definition of Whitney stratifications
and for a list of its main properties, and
to \cite[Part I, section 1.4]{GM} for further references on the above-mentioned result.

Another consequence of triangulability of complex-analytic spaces is that
they are locally contractible, and so for such spaces connectedness is equivalent to path-connectedness.\\
%
%A useful property of Whitney stratifications is that,
%if $Y$ is an analytic subset of $X$, 
%then $\{Y\cap X_i\}$ is a Whitney stratification of $Y$.
%

%
%\begin{rem}\label{rem:loc-triv}
%Let $Y$ be a reduced, irreducible subvariety of $X$,
%and let $\{X_i\}_{i\in I}$ be a good stratification of $X$
%such that $\mathring{Y}:=X_j$
%is Zariski-open (dense) in $Y$ for some $j\in I$.
%By the properties listed above,
%the neighbourhood $U_{\mathring{Y}/X}:=U_j$ and 
%the link $L_{\mathring{Y}/X}:=L_j$
%are fiber bundles over $\mathring{Y}$.
%
%Observe that, given {\it{any}} good stratification
%$\{X'_i\}_{i\in I}$ of $X$,
%the stratification $\{X'_i,\ Y\cap X'_i\}_{i\in I}$ is good,
%is a refinement of $\{X'_i\}_{i\in I}$
%and has a stratum which is a Zariski-open (dense) subset of $Y$.
%\end{rem}

It can be shown that,
given a collection $\{Y_\alpha\}$ of reduced analytic subspaces of $X$,
every good stratification of $X$
can be refined to a good stratification which is compatible with $\{Y_\alpha\}$ in the following sense, cf.~\cite{Ch}.

\begin{dfx}[Compatibility]
Let $\{Y_\alpha\}$ be a collection of reduced analytic subspaces of $X$.
A stratification of $X$ is
{\it{compatible with $\{Y_\alpha\}$}} if
every finite intersection 
$Y_{\alpha_1}\cap\dots\cap Y_{\alpha_k}$ is a union of strata.
\end{dfx}

\subsection{Irreducibility and fundamental group}\label{sec:cpx-irred}

Here we collect some remarks about fundamental groups
and local irreducibility of analytic varieties.
We underline that ``locally irreducible'' is to be understood with respect
to the classical topology, and that such condition is equivalent to being ``unibranch''
(i.e.~locally irreducible with respect to the \'etale topology).

The results collected in the following lemma are classical.

\begin{lm}[On irreducibility and local irreducibility]\label{lm:irred}
Let $X$ be a reduced analytic space. 
\begin{itemize}
\item[(i)]
$X$ is irreducible if and only if its smooth locus $X_{\sm}$ is connected.
\item[(ii)]
If $X$ is locally irreducible and connected, then it is irreducible.
%\end{itemize}
%Moreover,
%\begin{itemize}
\item[(iii)]
If $X$ is normal, then it is locally irreducible.
\item[(iv)]
%If $X$ is locally irreducible at $x\in X$, then the link of $x$ is connected.
If $X$ is locally irreducible along an irreducible locally trivial subspace $W\subset X$, then the bundles $L_{W/X}\rar W$ 
and $\dot{U}_{W/X}$ have connected fibers.
\item[(v)]
If $X$ is normal and irreducible (resp. locally irreducible and connected), then so is its universal cover.
\end{itemize}
\end{lm}
\begin{proof}
(i) is proven in \cite[page 55]{Ch}.

(ii) Since $X$ is reduced, for every open subset $V$ of $X$
the locus $V\cap X_{\sm}$ is dense in $V$.
By local irreducibility of $X$ and part (i), the subset $V\cap X_{\sm}$ is connected. The conclusion follows from the connectedness of $X$ and (i).

(iii) is proven in \cite{zariski}.

(iv) Since $X$ is locally irreducible along $W$, the subset $U_{W/X}$ is irreducible and so is $\dot{U}_{W/X}$. By (i) the smooth locus
of $\dot{U}_{W/X}$ is dense and connected. Since $\dot{U}_{W/X}\rar W$ is a locally trivial bundle, the smooth locus of every fiber is dense and connected, and so every fiber is connected. The same conclusion holds for
$L_{W/X}$, since $\dot{U}_{W/X}$ is homeomorphic to an $\RR$-bundle over $L_{W/X}$.

(v)
By hypothesis $X$ is connected and so its universal cover $\wti{X}$ is connected.
Normality and local irreducibility are local properties, so they are inherited by $\wti{X}$.
The conclusion follows by (ii) and (iii).
\end{proof}
%
%The following is then an immediate consequence.
%
%\begin{cor}[Irreducibility of universal cover]
%If an analytic space $X$ is normal and irreducible (resp. locally irreducible and connected), then so is its universal cover $\wti{X}$.
%\end{cor}
%\begin{proof}
%By hypothesis $X$ is connected and so $\wti{X}$ is connected.
%Normality and local irreducibility are local properties, so they are inherited by $\wti{X}$.
%The conclusion follows by Lemma \ref{lm:irred}(ii-iii).
%\end{proof}

We believe that the following should be well-known.
Since we have been unable to find a proper reference, we include a proof 
for completeness.

\begin{lm}[Fundamental groups of locally irreducible varieties]\label{lm:normal}
Let $X$ be a connected, locally irreducible analytic variety of positive dimension.
\begin{itemize}
\item[(i)]
If $W\subsetneq X$ is a closed analytic subspace, then
$\mathring{X}=X\setminus W$ is connected and
the inclusion $\mathring{X} \hookrightarrow X$ induces a surjection
$\pi_1(\mathring{X})\thra\pi_1(X)$.
\item[(ii)]
$X_{\sm}$ is connected and
$\pi_1(X_{\sm})\rar\pi_1(X)$ is surjective.
\end{itemize}
\end{lm}
\begin{proof}
Note first that $X$ is irreducible by Lemma \ref{lm:irred}(ii).

(i) Since $X$ is irreducible, $\mathring{X}$ is irreducible too and so it is connected.
Consider a good stratification of $X$ such that $W$ is a union of strata (see Section \ref{sec:stratification}) and let $W_d$ be the union of strata inside $W$ of dimension at most $d$.
Since $W=W_N$ for $N$ large and $W_{-1}=\emptyset$, it is enough to show
that $\pi_1(X\setminus W_d)\rar\pi_1(X\setminus W_{d-1})$ is surjective 
for all $d\geq 0$.

Fix $d\geq 0$ and let $X'=X\setminus W_{d-1}$ and $W'=W_d\setminus W_{d-1}$,
and let $\iota:X'\setminus W'\hra X'$ be the inclusion.
We need to show that $\iota_*:\pi_1(X'\setminus W')\thra \pi_1(X')$ (the basepoint being
picked anywhere in $X'\setminus W'$).
This is an immediate consequence of the fact that the fibers of $L_{W'/X'}\rar W'$ are connected. For sake of completeness, here we give a thourough proof.

Note that $W'$ is a closed subset of $X'$ and a disjoint union of locally trivial
smooth subvarieties of dimension $d$.
We recall that $U_{W'/X'}\rar W'$ and $\dot{U}_{W'/X'}\rar W'$ are locally trivial fibrations, though the topology of the fiber might depend on the connected component of $W'$.
Moreover, by Lemma \ref{lm:irred}(iv), the fibers of $\dot{U}_{W'/X'}\rar W'$ are connected. Hence, a path in $U_{W'/X'}$ that joins two points $x_0,x_1$ of $L_{W'/X'}$
is homotopic (inside $U_{W'/X'})$ to a path entirely contained inside $L_{W'/X'}$ that goes from $x_0$ to $x_1$.

Consider now the class in $\pi_1(X')$ of a loop $\alpha$: we want to show
that $[\alpha]$ is in the image of $\iota_*$.
By the above observation, any portion of $\alpha$ that is contained inside $U_{W'/X'}$
can be replaced by a path with the same endpoints and homotopic
to it, with support entirely
contained in $L_{W'/X'}$. As a result, we produce another loop $\alpha'$ in $X'$
homotopic to $\alpha$, whose support avoids $W'$, and so $[\alpha]=\iota_*[\alpha']$.

(ii) follows from (i) by taking $W=X_{\sing}$.
\end{proof}

In the following lemma we estimate the fundamental group of the smooth locus
of a variety $Z$ in terms of the fundamental group of a subvariety $Y\subset Z$,
thus employing the first technique mentioned at the beginning of the section.

\begin{lm}[Estimating the fundamental group of $Z_{\sm}$]\label{surjects}
Suppose that  $Z$ is  an irreducible analytic variety and $Y\subset Z$ an irreducible, locally irreducible subvariety with $\dim(Y)<\dim(Z)$.
Let $H$ be the image of $\pi_1(Y)\rar\pi_1(Z)$ and $K$ be the image of $\pi_1(Z_{\sm})\rar\pi_1(Z)$.
Then $H\cap K$ has finite index in $H$.
%If $Z$ has $k$ branches at the general point of $Y$, then $H\cap K$ has index at most $k$ in $H$.
\end{lm}

Note that, if $Z$ is not locally irreducible, the image of $\pi_1(Z_{\sm})$ need not be of finite index inside $\pi_1(Z)$.

\begin{proof}[Proof of Lemma \ref{surjects}]
Note preliminarly that $Z_{\sm}$ is smooth and connected
by Lemma \ref{lm:irred}(i), and so $Z_{\sm}\setminus \ol{Y}$ is connected
by Lemma \ref{lm:normal}(i).
A Whitney stratification of $Y$ inside $Z$ has one Zariski-open stratum
$\mathring{Y}$. Such $\mathring{Y}$ is then smooth and
locally trivial inside $Z$, and so it is connected by 
Lemma \ref{lm:irred}(i).
% whose link inside $Z$ is a fibre bundle $L_{\mathring{Y}/Z}$ with fiber $F$.
Observe that $\pi_1(\mathring{Y})$ maps surjectively onto $\pi_1(Y)$ by Lemma \ref{lm:normal}(i).

%(i) By hypothesis, up to restricting $\mathring{Y}$ to a smaller Zariski-open subset, we can assume that $Z$ is locally irreducible at all points of $\mathring{Y}$.
%Since $\mathring{Y}$ is locally irreducible
%and $Z$ is locally irreducible at points of $\mathring{Y}$, the fiber $F$ is connected.
%Now, $L_{\mathring{Y}/Z}$ can be embedded in $V\setminus Y$,
%where $V$ is a small classical neighbourhood of $Y$ in $Z$.
%We denote by $L_{\mathring{Y}/Z_{\sm}}:=L_{\mathring{Y}/Z}\setminus Z_{\sing}$ the locus of $L_{\mathring{Y}/Z}$ that corresponds to
%smooth points in $Z$. Since $Z$ is locally irreducible at points of $\mathring{Y}$, 
%the map $L_{\mathring{Y}/Z_{\sm}}\rar\mathring{Y}$ is a submersion of manifolds of finite type
%with connected fibers and so the composition $\pi_1(L_{\mathring{Y}/Z_{\sm}})\rar\pi_1(L_{\mathring{Y}/Z})\rar \pi_1(\mathring{Y})$ is surjective.
%The conclusion follows noting that $\pi_1(L_{\mathring{Y}/Z_{\sm}})\rar \pi_1(Z_{\sm})\rar\pi_1(Z)$
%factors through $\pi_1(L_{\mathring{Y}/Z_{\sm}})\rar \pi_1(\mathring{Y})\rar\pi_1(Z)$.
%
%(ii) 
%The argument is similar to the case of the  above cited lemma. 
Let $k$ be the number of branches of $Z$ at each point of $\mathring{Y}$,
that is the number of connected components of $\dot{U}=\dot{U}_{\mathring{Y}/Z}$.
Pick a connected component $\dot{U}'$ of $\dot{U}$.
The fiber $F'$ of $\dot{U}'\rar \mathring{Y}$ has $c\leq k$ connected components and the exact sequence
$\pi_1(\dot{U}')\rar \pi_1(\mathring{Y})\rar \pi_0(F')\rar \{\ast\}$ 
shows that the image of $\pi_1(\dot{U}')\rar \pi_1(\mathring{Y})$ has index $c$. 

%Moreover $\dot{U}'_{\mathring{Y}/Z}$
%corresponds to the choice of $c$ branches $Z_1,\dots,Z_c$ of $Z$ at $\mathring{Y}$.
%Let $V_i$ be a small classical neighbourhood of $\mathring{Y}$ inside $Z_i$ and let $V'=\bigcup_i V_i$
%Then $L'_{\mathring{Y}/Z}$ can be embedded inside $V'\setminus Y$.

Let $\dot{U}'_{\sm}:=\dot{U}'\cap Z_{\sm}$.
The fibers of $\dot{U}'_{\sm}\rar\mathring{Y}$ are dense in those of $\dot{U}'\rar\mathring{Y}$ and so they have $c$ connected components too.
As a consequence, 
the images of $\pi_1(\dot{U}'_{\sm})\thra \pi_1(\mathring{Y})$
and of $\pi_1(\dot{U}')\thra \pi_1(\mathring{Y})$ coincide.

The conclusion follows from the commutativity of the diagram
\[
\xymatrix{
\pi_1(\dot{U}'_{\sm}) \ar[rr] \ar@{^(->}[d]_{\text{index $c$}}\ar[rrd] && \pi_1(Z_{\sm}) \ar[d] \\
\pi_1(\mathring{Y}) \ar@{->>}[r] & \pi_1(Y) \ar[r] & \pi_1(Z)
}
\]
since the image of $\pi_1(\dot{U}'_{\sm})\rar\pi_1(Z)$ is contained inside $H\cap K$ and has index
at most $c$ in $H$.
\end{proof}

Here is a very useful consequence of the above lemma.

\begin{cor}[Image of fundamental groups of smooth loci]\label{cor:image-smooth}
Let $Y$ and $Z$ be irreducible algebraic varieties
and let $f:Y\rar Z$ be morphism which is generically finite onto its image.
Let  $H$  be the image of $\pi_1(Y_{\sm})\rar\pi_1(Z)$ and $K$ the image of $\pi_1(Z_{\sm})\rar\pi_1(Z)$.
Then
\begin{itemize}
\item[(i)]
$K\cap H$ is a finite-index subgroup of $H$.
%
%$[H:K\cap H]\leq dk$, where $k$ is the number of branches of $Z$ at the general point of $f(Y)$.
\item[(ii)]
If $\dim(Y)=\dim(Z)$, then $H$ is a finite-index subgroup of $K$. 
%$[K:H]\leq d$.
\end{itemize}
\end{cor}
\begin{proof}
Let $d$ be the cardinality of the general fiber of $f:Y\rar f(Y)$
and note that $f(Y)$ is irreducible.
Let $W\subset f(Y)$ be the union of $f(Y)_{\sing}$,
of $f(Y_{\sing})$ and of the locus of $z\in f(Y)$
such that $f^{-1}(z)$ does not consist of $d$ distinct points.
Up to taking a larger $W$, we can assume that $f(Y)\setminus W$
is locally trivial.

Since $Y,Z,f(Y)$ are irreducible, their smooth loci are connected, and so the smooth
$\mathring{Y}=Y\setminus f^{-1}(W)$ and $f(\mathring{Y})=f(Y)\setminus W$ are connected.
Moreover, the restriction $\bar{f}:\mathring{Y}\rar f(\mathring{Y})$ of $f$ is a topological cover of degree $d$.

Observe that the diagram
\[
\xymatrix{
\pi_1(\mathring{Y}) \ar@{->>}[rr] \ar[rrrd] \ar@{^(->}[d]^{\bar{f}_*}_{\text{index $d$}} && \pi_1(Y_{\sm})\ar[rd] \\
\pi_1(f(\mathring{Y}))\ar[rrr] &&&  \pi_1(Z) 
}
\]
is commutative
and let $\check{H}$ be the image of $\pi_1(f(\mathring{Y}))\rar\pi_1(Z)$.
Since $H$ coincides with the image of $\pi_1(\mathring{Y})\rar\pi_1(Z)$, the subgroup
$H$ is contained inside $\check{H}$ with $[\check{H}:H]\leq d$.

(i) Suppose first that $\dim(Y)<\dim(Z)$
and let $k$ be the number of connected components of the fibers of
$\dot{U}_{f(\mathring{Y})/Z}\rar f(\mathring{Y})$.
By Lemma \ref{surjects} applied to $f(\mathring{Y})\subset Z$,
the subgroup $K\cap \check{H}$ has finite index in $\check{H}$: in particular,
the proof of Lemma \ref{surjects} shows that
$[\check{H}:K\cap\check{H}]\leq k$. It follows that $[H:K\cap H]\leq 
[H:K\cap\check{H}]=[H:\check{H}]\cdot [\check{H}:K\cap\check{H}]\leq dk$.

(ii) Suppose now that $\dim(Y)=\dim(Z)$.
Since $Z$ is irreducible, $f(Y)$ is a Zariski-dense open subset of $Z$
and so $\mathring{Y}:=f^{-1}(Z_{\sm})$ is a Zariski-open subset of $Y$.
Thus $\pi_1(\mathring{Y})\rar\pi_1(Z)$ 
factors through the surjective map $\pi_1(f(\mathring{Y}))\thra\pi_1(Z_{\sm})$.
It follows that $\check{H}=K$ and so $H\subseteq K$.
Hence, $[K:H]=[\check{H}:H]\leq d$.
\end{proof}

\subsection{Liftings and connectedness}\label{sec:connected}

Let $X,Z$ be connected and locally arc-connected topological spaces 
with universal covers $p:\wti{X}\rightarrow X$ and $\univ{Z}\rar Z$.
Given a map $f:Z\rar X$, consider the following diagram
\[
\xymatrix{
{\univ{Z}} \ar[d]  \ar@/^1pc/[drr]^{\univ{f}_i}\\
\wti{Z}_i \ar[dr] \ar@{^(->}[r] & Z\times_X \wti{X} \ar[r]^{\quad\tilde{f}} \ar[d]^{\tilde{p}} & \wti{X} \ar[d]^{p}\\
& Z \ar[r]^f & X
}
\]
where the rectangle is Cartesian and $\wti{Z}_i$ is a connected component of  $Z\times_X \wti{X}$. We denote by $f_*:\pi_1(Z)\rar\pi_1(X)$ the homomorphism induced by $f$.

\begin{lm}[Liftings]\label{lm:lift}
Let $f:Z\rightarrow X$ be a map of connected topological spaces.
\begin{itemize}
\item[(i)]
The  group $\pi_1(X)$  acts   transitively on the set of liftings $\univ{f}:\univ{Z}\rightarrow\wti{X}$.
\item[(ii)]
The $\pi_1(X)$-set of connected components of
the fiber product $Z\times_X \wti{X}$ is isomorphic to $\pi_1(X)/f_*\pi_1(Z)$.
For each component $\wti{Z}_i$ of $Z\times_X \wti{X}$, its fundamental group satisfies
$\pi_1(\wti{Z}_i)=\ker(f_*)$ and there exists a lift $\univ{f}_i$ as in (i)
that factors through the cover $\univ{Z}\rightarrow \wti{Z}_i$.
\item[(iii)]
$Z\times_X \wti{X}$ is connected if and only if $f_*:\pi_1(Z)\rightarrow\pi_1(X)$ is surjective.
%\item[(iv)]
%If $Z$ is locally irreducible and $f$ is an embedding,
%then the lifts $\wh{f}_i$ have distinct images.
\end{itemize}
\end{lm}

The above statement follows from standard arguments in the theory
of covering spaces. We include a proof for convenience.

\begin{proof}[Proof of Lemma \ref{lm:lift}]
Fix a point $z\in Z$ and choose $\univ{z}\in\univ{Z}$ a lift of $z$.
Let also $x=f(z)\in X$ and choose $\tilde{x}\in \wti{X}$ a lift of $x$.

(i) A lifting $\univ{f}$ is uniquely determined by the choice of $\univ{f}(\univ{z})\in p^{-1}(x)$.
The conclusion follows since the set $p^{-1}(x)$ is acted on simply transitive by $\pi_1(X,x)$.

(ii) The group $\pi_1(X,x)$ acts by simply and transitively permuting the elements
in $p^{-1}(x)$, and so $\pi_1(X,x)\cdot\tilde{x}=p^{-1}(x)$.
It can be easily seen that the surjective map $\pi_1(X,x)\cdot (z,\tilde{x})=\tilde{p}^{-1}(z)\rar \pi_0(Z\times_X\wti{X})$
is a map of $\pi_1(X,x)$-sets.
Moreover,
two elements $(z,\tilde{x}),(z,\tilde{x}')$ 
in $\tilde{p}^{-1}(z)$
belong to the same
connected component of $Z\times_X \wti{X}$ if and only if there exists a path $\alpha\in\pi_1(Z,z)$
such that $f\circ\alpha$ lifts to a path in $\wti{X}$ that joins $\tilde{x}$ and $\tilde{x}'$,
namely $f_*(\alpha)\cdot \tilde{x}=\tilde{x}'$. Hence, $\pi_0(Z\times_X\wti{X})$
can be identified to $\pi_1(X,x)/f_*\pi_1(Z,z)$.

By the universal property, any lifting $\tilde{f}$ factors through $Z\times_X \wti{X}$ and covers a connected component of $Z\times_X\wti{X}$. If the component $\wti{Z}_i$ contains the point $\tilde{z}_i:=(z,\tilde{x})$, then the lift that satisfies $\univ{f}(\univ{z})=\tilde{x}$ induces a cover $\univ{Z}\rightarrow \wti{Z}_i$.

Pick a connected component $\wti{Z}_i$ of $Z\times_X \wti{X}$. It covers $Z$, and so
$\pi_1(\wti{Z}_i,\tilde{z}_i)\hookrightarrow\pi_1(Z,z)$ is injective.
Moreover, the composition $\pi_1(\wti{Z}_i,\tilde{z}_i)\rightarrow\pi_1(Z,z)\rightarrow \pi_1(X,x)$ vanishes,
as it factors through $\pi_1(\wti{X},\tilde{x})=\{e\}$, and so $\pi_1(\wti{Z}_i,\tilde{z}_i)$ injectively
maps into $\ker(f_*)$. It also maps surjectively, since all elements of $\ker(f_*)$ lift
to closed loops in $\wti{Z}_i$.

(iii) follows from (ii).
%
%(iv) Since $Z$ is locally irreducible and connected, it is also irreducible
%and so its smooth locus $U=Z_{\sm}$ is connected.
%Moreover $\pi_1(U)\thra\pi_1(Z)$ and so
%the restriction $g=f|_U:U\hra X$
%satisfies $\mathrm{Im}(g_*)=\mathrm{Im}(f_*)\subseteq\pi_1(X)$.
%Hence, it is enough to prove the statement
%for the map $g$.
%
%Consider then two lifts $\wti{U}_i$ and $\wti{U}_j$.
%Let $\gamma\in\pi_1(X)$ such that $\wti{U}_j=\gamma\cdot\wti{U}_i$.
%Both $\wti{U}_i,\wti{U}_j$ map isomorphically onto their images
%$\wti{U}_i,\wti{U}_j$ inside $\wti{X}$.
%Suppose now they have the same image.
%Then $\gamma$ induces an automorphism of $\wti{U}_i$,
%which is actually an automorphism of the cover $\wti{U}_i\rar U$,
%because $g$ is an embedding.
%Hence there is an element in $\pi_1(U)$
%that is mapped to $\gamma$ via $g_*$,
%and so $\gamma\in g_*\pi_1(U)$. 
%It follows that $\wti{U}_i=\wti{U}_j$.
\end{proof}

An argument analogous to Lemma \ref{lm:lift}(ii) also shows the following.

\begin{lm}[Liftings via finite covers]\label{lm:finite-lift}
Let $p':X'\rar X$ be a covering space
and $f:Z\rar X$ a map of connected topological spaces.
\begin{itemize}
\item[(i)]
%the $G$-set of
%the connected components of $Z\times_X X'$ is in bijective
%correspondence with $G/G'$
%\item[(ii)]
If $p'$ is a finite cover, then
the number of connected components of $Z\times_X X'$ is 
$[\pi_1(X)\,:\,p'_*\pi_1(X')]/[f_*\pi_1(Z)\,:\,p'_*\pi_1(X')\cap f_*\pi_1(Z)]$.
%where $G=\pi_1(X)$, $G'=p'_*\pi_1(X')$ and $H=f_*\pi_1(Z)$. 
\item[(ii)]
$Z\times_X X'$ is connected if and only if
the induced map $\pi_1(Z)\rar \pi_1(X)/p'_*\pi_1(X')$ is surjective.
\end{itemize}
\end{lm}

Out of the above lemmas, we can already draw a first consequence,
which is an example of the strategy outlined at the beginning of the section.

\begin{cor}[Connected lifting of analytic subspaces]\label{cor:connected-LHT}
Let $X'\rar X$ be a cover of connected complex-analytic spaces, $Z\subset X$ a connected analytic subspace
and $Z'$ its preimage inside $X'$.
%\begin{itemize}
%\item[(i)]
Suppose that there exists an analytic subspace $Y\subseteq Z$
such that $\pi_1(Y)\thra\pi_1(X)$. Then $Z'$ is connected.
%\item[(ii)]
%Suppose that $X$ admits a compactification $\ol{X}$ such that
%$(\ol{X},\pa X,\ol{Z})$ satisfy properties (I)-(III$^{ci}$)-(III$^{ci}_2$)
%in Section \ref{sec:LHT}. Then $Z'$ is connected.
%\end{itemize}
\end{cor}
\begin{proof}
In view of Lemma \ref{lm:lift}, it is enough to show that $\pi_1(Z)\rar \pi_1(X)$ is surjective.
This is immediate because of the factorization $\pi_1(Y)\rar\pi_1(Z)\rar\pi_1(X)$.
%
%(ii) It follows from Corollary \ref{cor:LHT}.
\end{proof}

\subsection{Liftings and irreducibility}\label{sec:lift}

Let $X,Z$ be as in Section \ref{sec:connected}
and assume that both $X$ and $Z$ are
connected analytic spaces.

We begin by stating our fundamental irreducibility criterion,
which relies on Lemma \ref{lm:lift} and Lemma \ref{lm:finite-lift}.

\begin{cor}[Irreducibility criterion of $Z\times_X \wti{X}$]\label{cor:criterion}
Let $f:Z\rightarrow X$ be a map of complex-analytic varieties, with $Z$ irreducible.
\begin{itemize}
\item[(i)]
The analytic space $Z\times_X \wti{X}$ is irreducible if and only if $f_*:\pi_1(Z_{\sm})\rar\pi_1(X)$ is surjective.
\item[(ii)]
If $p':X'\rar X$ is a covering space, then the irreducible components of $Z\times_X X'$
correspond bijectively  to the cokernel of the map $\pi_1(Z_{\sm})\rar \pi_1(X)/p'_*\pi_1(X')$ induced by $f_*$.
\end{itemize}
\end{cor}
\begin{proof}
(i) It is enough to observe that $Z_{\sm}\times_X\wti{X}$ is smooth and Zariski-dense in $Z\times_X\wti{X}$
and that $Z_{\sm}\times_X\wti{X}$ is irreducible if and only if it is connected.
The conclusion follows from Lemma \ref{lm:lift}(iii).

(ii) is analogous to (i), by applying Lemma \ref{lm:finite-lift} instead of Lemma \ref{lm:lift}(iii).
\end{proof}

The following proposition is an incarnation of the strategy mentioned at the beginning of the section, namely
to prove the irreducibility of $Z\times_X\wti{X}$ by using certain subvarieties of
$Z$ with large fundamental group.

\begin{prop}[Irreducibility of $Z\times_X \wti{X}$ via subvarieties of $Z$]\label{prop:irred}
Let $Z$ be an irreducible complex-analytic variety, and let $f:Z\rar X$ be a morphism.
\begin{itemize}
\item[(i)]
Suppose that $Z$ is locally irreducible and 
 $f_*:\pi_1(Z)\rightarrow \pi_1(X)$ is surjective.
Then $Z\times_X\wti{X}$ is irreducible.
\end{itemize}
Let now $Y\subset Z$ be an irreducible and locally irreducible subvariety.
\begin{itemize}
\item[(ii)]
If $\pi_1(Y)\rar\pi_1(X)$ is surjective,
then $Z\times_X\wti{X}$ has finitely many irreducible components.
\item[(iii)]
If the image of $\pi_1(Y)\rar\pi_1(X)$ has finite index,
and $Z\times_X X'$ is irreducible for all finite \'etale covers $p':X'\rar X$,
then $Z\times_X \wti{X}$ is irreducible.
\end{itemize}
\end{prop}
\begin{proof}
By Lemma \ref{lm:irred}(i), the smooth locus $Z_{\sm}$ is connected.

(i) By Lemma \ref{lm:normal}(ii), the map
$\pi_1(Z_{\sm})\thra\pi_1(Z)$ is surjective.
Thus the composition $\pi_1(Z_{\sm})\rightarrow\pi_1(Z)\rar \pi_1(X)$ is surjective too.
We conclude by Corollary \ref{cor:criterion}(i).

%(ii) Clearly, $\pi_1(\mathring{Y})\thra\pi_1(Y)$ by Lemma \ref{lm:normal}(i), where $\mathring{Y}:=Y\cap Z_{\sm}$.
%It follows that $\pi_1(\mathring{Y})\thra\pi_1(X)$ and so $\pi_1(Z_{\sm})\thra\pi_1(X)$.
%The conclusion follows by Corollary \ref{cor:criterion}(i).

(ii) Since $\pi_1(Y)\rar\pi_1(X)$ is surjective, so is $\pi_1(Z)\rar\pi_1(X)$.
Let $k$ be the number of branches of $Z$ at the general point of $Y$.
By Lemma \ref{surjects}, the image of $\pi_1(Z_{\sm})\rar\pi_1(Z)$ has index at most $k$, and so the same holds for the image of $\pi_1(Z_{\sm})\rar\pi_1(X)$.
The conclusion then follows from Corollary \ref{cor:criterion}(ii) applied to the universal cover $\wti{X}\rar X$.

(iii) Since $Z$ has finitely many branches at the general point of $Y$,
the image of $\pi_1(Z_{\sm})\rar\pi_1(X)$ has finite index by Lemma \ref{surjects}: let $d$ be such index.
Let $p':X'\rar X$ be the \'etale cover of degree $d$ such that $p'_*\pi_1(X')=f_*\pi_1(Z_{\sm})$.
By Corollary \ref{cor:criterion}(ii) the fiber product $Z\times_X X'$ has $d$ irreducible components,
and our hypothesis implies that $d=1$. It follows that $\pi_1(Z_{\sm})\rar\pi_1(X)$ is surjective and we conclude 
by Corollary \ref{cor:criterion}(i).
\end{proof}

The following special case is a direct consequence
of Proposition \ref{prop:irred}(ii).

\begin{cor}[A criterion of irreducibility of $Z\times_X\wti{X}$]\label{cor:irred-crit}
Let $Z$ be an irreducible complex-analytic variety
and $f:Z\rar X$ be a morphism.
Suppose that there exists an irreducible and locally irreducible subvariety $Y\subset Z$
such that $\pi_1(Y)\thra\pi_1(X)$.
Then $Z\times_X\wti{X}$ has finitely many irreducible components.\\
Moreover, if $Y$ intersects the smooth locus of $Z$, then $Z\times_X\wti{X}$ is irreducible.
\end{cor}

%
%
%We conclude this subsection by discussing how to deduce irreducibility
%of the preimage of an ample divisor following the second strategy, namely employing
%a hyperplane section theorem.
%
%
%\begin{prop}
%Let $\ol{X}$ be a variety and let $\pa X$ and $\ol{D}$ be subschemes of $\ol{X}$
%that satisfy properties (I)-(II$_2$)-(II)-(III$_2$) listed in Section \ref{sec:LHT}.
%Then the preimage $\wti{D}=p^{-1}(D)$ of $D=\ol{D}\setminus\ol{D}$
%via the universal cover $p:\wti{X}\rar X$ is irreducible.
%\end{prop}
%

\subsection{The case of orbispaces}

Most of the above results hold in the category of complex-analytic
orbispaces, namely objects locally modelled on $[T/G]$, where $T$ is a complex-analytic space
and $G$ is a finite group that acts on $T$ via biholomorphisms).
In this case, we must use open orbifold charts instead of open subsets, and we must modify the
definition of neighborhoods accordingly. Moreover, the word smooth, singular, unramified cover,
fiber bundle, universal cover and fundamental group must all be understood in the orbifold sense.
Note in particular that an orbispace is smooth where it is locally modelled as $[T/G]$ with $T$ smooth.
An analogous interpretation must be reserved to the words normal, irreducible, locally irreducible.

With the above caveat, the results
in Sections \ref{sec:stratification}-\ref{sec:cpx-irred}-\ref{sec:connected}-\ref{sec:lift} still hold in the setting of complex-analytic orbispaces.

   %%%%%%%%%%%%%%%%%%%%%%%%%%%%%%%%%%%%%%%%%%%%%%%%%%%%%%%%%%%%%%%%%%%%%%%
%%%%%%%%%%%%%%%%%%%%%%%%%%%%%%%%%%%%%%%%%%%%%%%%%%%%%%%%%%%%%%%%%%%%%%%%%%%

\section{A hyperplane section theorem}\label{sec:LHT}

In the following section we present a generalization (Theorem \ref{thm:divisor}) of the
homotopical Lefschetz hyperplane section theorem (LHT) for $\pi_0$ and $\pi_1$
to smooth quasi-projective varieties that admit a projective model with small boundary. Such result then extended to complete intersections (Corollary \ref{cor:LHT}) and to complete intersections inside orbifolds that are global quotients (Corollary \ref{cor:LHT-orbifold}).\\

We will often mention the following properties of
a variety $\ol{X}$ and of its algebraic loci $\pa X$ and 
$\ol{D}$.
\begin{itemize}
\item[(I)]
$\ol{X}$ is a connected projective variety of dimension $N$
and $\pa X$ be a closed subscheme of dimension at most $N-1$\
such that $X=\ol{X}\setminus\pa X$ is smooth and connected %\item[(II$_h$)]
%$\mathrm{dim}(\ol{X})=N\geq h+1$
%\item[(II)]
%$\ol{D}\subset \ol{X}$ is an effective Cartier divisor
\item[(II)]
$\ol{D}\subset \ol{X}$ is the support of an effective, ample Cartier divisor
% hypersurface, dual in $H^2(X;\QQ)$ to the restriction of ample class from $H^2(\ol{X};\QQ)$
\item[(III$_h$)]
$\mathrm{codim}(\pa D/\ol{D})\geq h$, where $\pa D=\ol{D}\cap\pa X$.
\end{itemize}
%Condition (II$_h$) will ensure that we are in a nontrivial situation.
We remark that (III$_h$) is certainly implied by
\begin{itemize}
\item[(III$'_h$)]
$\mathrm{codim}(\pa X/\ol{X})\geq h+1$.
\end{itemize}

The version of LHT we wish to prove is the following.
In the typical application that we have in mind
we will pick as $X$ the moduli space $\calA_g(n)$ for some level $n\geq 3$, as $\ol{X}$ its Satake compactification and as 
$D$ the zero locus of a modular form.

\begin{thm}[Improved LHT for smooth quasi-projective varieties]\label{thm:divisor}
Let $\ol{X}$ be a variety and $\pa X$, $\ol{D}$ be subschemes of $\ol{X}$ such that
properties (I)-(II)-(III$_h$) above hold with $h=2$ or $h=3$.
Then $D=\ol{D}\setminus\pa D$ is connected
and the natural  map $\pi_1(D)\rightarrow\pi_1(X)$ is an isomorphism if $h=3$ (resp. is surjective, if $h=2$).
\end{thm}

The key ingredient of the proof is the 
Lefschetz hyperplane type theorem (LHT) that appears at the very beginning of Part II, Section 5.1 of \cite{GM}.
In particular, we will use the two versions of such theorem
(see the ``furthermore'' below the statement of (LHT))
for general hyperplane (LHT-gen) and for compact hyperplane section (LHT-cpt). A similar idea is already in \cite[Theorem 8.7]{amoros}. 

%
%Observe that the statement of Theorem \ref{thm:divisor}
%is just obtained from the statement of Theorem \ref{thm:hyperplane}
%by weakening the transversality hypothesis (b):
%this can be done if $\pa X$ is small, as was already observed in
%

%
%endowed with a stratification 
%and let $\pa X$ be a closed union of strata of $\ol{X}$ and $X=\ol{X}\setminus\pa X$.
%

\begin{proof}[Proof of Theorem \ref{thm:divisor}]
%Endow $\ol{X}$ with a good stratification that is compatible with $\pa X$ and $\ol{D}$
%in the sense of Section \ref{sec:stratification}.
Consider then general very ample divisors $\ol{D}_1,\dots,\ol{D}_{N-h}$ in $\ol{X}$
such that
\begin{itemize}
\item[(a)]
$D_i:=\ol{D}_i\setminus \pa D_i$ is smooth, where $\pa D_i=\ol{D}_i\cap \pa X$ 
%\item[(b)]
%$\ol{D}_1$ intersects all the strata of $\ol{D}$ and all strata of $\ol{X}$ transversally
%and $\ol{D}_k$ transversally
%intersects all strata of $\ol{D}\cap \ol{D}_1\cap\dots\ol{D}_{k-1}$
%and all strata of $\ol{D}_1\cap\dots\ol{D}_{k-1}$
%for all $k=2,\dots,N-h$
\item[(b)]
the intersection $E=D_1\cap\dots D_{N-h}$ is transverse, and so $E$
is a smooth variety of dimension $c$
\item[(c)]
$S=E\cap D=E\cap\ol{D}$ is compact, of dimension $h-1$
\item[(d)]
the singular locus of $S$ is contained in the singular locus of $D$.
\end{itemize}
Note that, in order to establish the existence of $S$, we use property (I) in (a)-(b)-(d), property (III$_h$) in (c).
%We are going to apply the first theorem that appears in
%Part II, Section 5.1 of \cite{GM} in the version
%Theorem \ref{thm:hyperplane} a few times
%in the case of $\ol{Z}$ an ample hypersurface.

Suppose first that $h=3$.
Properties (II) and (b-c) together with (LHT-cpt)
%Theorem \ref{thm:hyperplane}(a) 
imply that $S$ is a compact connected surface
and $\pi_1(S)\rightarrow\pi_1(E)$ is an isomorphism.

On the other hand, (b-d) and 
(LHT-gen)
%Theorem \ref{thm:hyperplane}(b)
imply that $E$ is connected and
\[
\pi_1(E)\rightarrow \pi_1(D_1\cap\dots D_{N-h-1})
\rightarrow\dots\rightarrow\pi_1(D_1\cap D_2)\rightarrow\pi_1(D_1)\rightarrow\pi_1(X)
\]
are isomorphisms.
Similarly, $S$ intersects every connected component of $D$ and 
\[
\pi_1(S)\rightarrow \pi_1(D\cap D_1\cap\dots D_{N-h-1})
\rightarrow\dots\rightarrow\pi_1(D\cap D_1)\rightarrow\pi_1(D)
\]
are isomorphisms.
It follows that $D$ is connected
and $\pi_1(S)\rightarrow\pi_1(D)\rightarrow\pi_1(X)$ are isomorphisms.\\

If $h=2$,
then we can argue analogously as above and consider
general very ample divisors $\ol{D}_1,\dots,\ol{D}_{N-2}$
such that $E=D_1\cap\dots D_{N-2}$ is a smooth surface and
$C=D\cap E$ is a compact curve.
A similar repeated application of (LHT-gen) and then of (LHT-cpt)
%Theorem \ref{thm:hyperplane}
gives
\[
\xymatrix{
\pi_1(C)\ar@{->>}[r] \ar@{->>}[d] & \pi_1(D) \ar[d]\\
\pi_1(E) \ar[r]^{\cong\ } & \pi_1(X),
}
\]
from which we conclude that $\pi_1(D)\rar\pi_1(X)$ is surjective.
\end{proof}

An analogous statement holds for complete intersections
$\ol{Z}$, by replacing properties (III)-(III$_h$) by
\begin{itemize}
\item[(III$^{ci}$)]
$\ol{Z}\subset\ol{X}$ is the support of a complete intersection of Cartier divisors,
whose classes in $X$ are all proportional to the same ample class in $H^2(X;\QQ)$,
\item[(III$^{ci}_h$)]
$\mathrm{codim}(\pa Z/\ol{Z})\geq h$, where $\pa Z=\ol{Z}\cap\pa X$.
\end{itemize}

More precisely, we have the following.

\begin{cor}[Improved LHT for complete intersections]\label{cor:LHT}
Suppose that $\ol{X}$, $\pa X$ and $\ol{Z}$ satisfy (I)-(III$^{ci}$)-(III$^{ci}_h$) above
for $h=2$ or $h=3$. Then $Z=\ol{Z}\setminus\pa Z$ is connected and
$\pi_1(Z)\rar\pi_1(X)$ is an isomorphism if $h=3$ (resp. is surjective, if $h=2$).
\end{cor}
\begin{proof}
By (III$^{ci}$) it is possible to embed $\ol{X}$ inside some projective space $\PP$
in such a way that $\ol{Z}=\ol{X}\cap H$ for some linear subspace $H\subset\PP$
of codimension $\codim(H/\PP)=\codim(\ol{Z}/\ol{X})$.
We then proceed analogously to the proof of Theorem \ref{thm:divisor},
replacing the role of $\ol{D}$ there by $\ol{Z}$.
\end{proof}

\begin{rem} Assuming that  $\ol{X}$, $\pa X$ and $\ol{Z}$ satisfy (I)-(III$^{ci}$)-(III$^{ci}_h$) above, it is  immediate  that for higher values of $h$ we have similar results about the higher homotopy groups.  In  particular we obtain that
 $\pi_i(Z)\rar\pi_i(X)$ is an isomorphism if $ i\leq h-2$, and $\pi_{h-1}(Z)\rar\pi_{h-1}(X)$ is surjective.
\end{rem}

\subsection*{The case of orbispaces}%\label{sec:orbispaces}

A version of LHT as in Section \ref{sec:LHT}
can be also phrased for orbifolds that are global quotients.
Let $\ol{X}$ be an analytic space and let $\pa X$ and $\ol{Z}$ be subspaces of $\ol{X}$.
Suppose moreover that a finite group $G$ acts on $\ol{X}$, preserving
$\pa X$ and $\ol{Z}$, and so preserving $Z=\ol{Z}\setminus\pa X$.

\begin{cor}[Improved LHT for c.i.~inside global quotients]\label{cor:LHT-orbifold}
Suppose that $\ol{X}$, $\pa X$ and $\ol{Z}$ satisfy properties (I)-(III$^{ci}$)-(III$^{ci}_h$)
in Section \ref{sec:LHT} with $h=2$ or $h=3$.
Then 
\begin{itemize}
\item[(i)]
the locus $[Z/G]$ inside the orbifold $[X/G]$ is connected
and the homomorphism $\pi_1([Z/G])\rar\pi_1([X/G])$ 
of orbifold fumdamental groups is an isomorphism if $h=3$ (resp. is surjective, if $h=2$);
\item[(ii)]
the preimage of $[Z/G]$ via an unramified cover of $[X/G]$ is connected.
\end{itemize}
\end{cor}
\begin{proof}
Both claims follow from Corollary \ref{cor:LHT}.
% and claim (ii) from Corollary \ref{cor:connected-LHT}(ii).
\end{proof}

\section{Absolutely irreducible divisors}\label{sec:absol}

Recall that in Section \ref{sec:main} we called a modular form ``absolutely irreducible'' if it cannot be written as a product of two nonconstant modular forms. Such definition motivates the following generalizations.

%
%The second hypothesis in Proposition \ref{prop:irred}(iii) determines a class of subvarieties in $X$
%that we will call ``absolutely irreducible'' (following Freitag).

\begin{dfx}[Absolutely irreducible and universally irreducible subvarieties]
A subvariety $Z$ of $X$ is {\it{absolutely irreducible}}
if its preimage through every finite \'etale cover $X'\rar X$ is irreducible.
Such $Z$ is {\it{universally irreducible}} if its preimage through the universal
cover $\wti{X}\rar X$ is irreducible.
\end{dfx}

Obviously, a universally irreducible subvariety is absolutely irreducible.
Below we exhibit a sufficient condition for an absolutely irreducible
subvariety to be universally irreducible.

The key observation is the following.

\begin{prop}[When $\wti{Z}$ has finitely many irreducible components]\label{prop:finitely-many-irr}
If the preimage $\wti{Z}$ of $Z\subset X$ inside $\wti{X}$
has finitely many irreducible components, then
there exists a finite \'etale cover $X'\rar X$ such that the preimage $Z'$
of $Z$ inside $X'$ has finitely many universally irreducible components.
\end{prop}
\begin{proof}
Decompose then $\wti{Z}$ into the union
\[
\wti{Z}=\bigcup_{i\in I}\wti{Z}_i
\]
of its irreducible components $\wti{Z}_i$.

The group $\Lambda:=\pi_1(X)$ acts on $\wti{Z}$ via biholomorphisms, and permutes its irreducible components.  
Hence we get a homomorphism
\[
\rho:\Lambda\lra \SY(I),
\]
where $\SY(I)$ is the symmetric group on $I$.
Consider the normal subgroup $\Lambda'=\mathrm{ker}(\rho)$ of $\Lambda$.
Since $I$ is finite, $\Lambda'$ has finite index in $\Lambda$
and so $X':=\wti{X}/\Lambda'$ is a finite \'etale cover of $X$.
Let $Z':=\wti{Z}/\Lambda'$ the subvariety of $X'$
obtained as inverse image of $Z$ through the cover $X'\rar X$,
and let $Z'_i:=\wti{Z}_i/\Lambda'$.
Clearly, $Z'_i$ is universally irreducible and so
$Z'=\bigcup_{i\in I}Z'_i$ is the decomposition of $Z'$ into its finitely many universally irreducible components.
\end{proof}

%We consider the following property of a complex variety $X$:
%\begin{itemize}
%\item[(IV)]
%the Picard number of $\ol{X}$ is $1$;
%there exists an ample class $\lambda\in H^2(\ol{X};\QQ)$ such that
%$H^2(X;\QQ)=\QQ\cdot\lambda$ 
%and $H^4(X;\QQ)=\QQ\cdot\lambda^2$
%\item[(IV)]
%the normal subgroups of $\pi_1(X)$ are either finite or of finite index,
%and $\pi_1(X)$ has finite-index subgroups of arbitrarily high index
%(the latter condition is implied by $\pi_1(X)$ being residually finite).
%\end{itemize}
%
%Note that, if $n\geq 3$, then every finite \'etale cover $X$ of $\calA_g(n)$
%satisfies property (IV) above. In fact, it also holds for $X$ any finite \'etale (in the orbifold sense) cover of $\calA_g$,
%provided $\pi_1$ is intended in the orbifold sense.

\begin{cor}[Promoting absolutely irreducible to universally irreducible]\label{cor:promoting}
Let $X$ be a connected variety, with universal cover $\wti{X}$,
and let $Z\subset X$ be an absolutely irreducible subvariety
whose inverse image $\wti{Z}\subset \wti{X}$ has a finite number of irreducible components. Then $Z$ is universally irreducible.
\end{cor}
\begin{proof}
By Proposition \ref{prop:finitely-many-irr} there exists a finite \'etale cover
$X'\rar X$ such that the preimage $Z'$ of $Z$ inside $X'$ is the union of finitely many universally irreducible subvarieties.
Since $Z$ is absolutely irreducible, such $Z'$ is itself irreducible.
It follows that $\wti{Z}$ is irreducible.
\end{proof}

 \begin{rem}
 We observe that in this case we do not require any hypothesis on $\pi_1(Z)$ and  we do not assume that $Z$ is locally irreducible.
 \end{rem}
 
Here is an immediate consequence of the above observations.
 
%We preliminarly observe that an irreducible subvariety $Z$ in $\calA_g(\Gamma)$ 
%is absolutely irreducible
%if and only if the image of $\pi_1(Z_{\sm})\rar\Gamma$
%maps onto every finite quotient of $\Gamma$.

\begin{cor}[Absolutely irreducible plus finite index implies universally irreducible]\label{cor:total}
Let $Z$ be a subvariety of $X$ such that the image of $\pi_1(Z_{\sm})\rar\pi_1(X)$ has finite index.
Then the preimage $\wti{Z}$ of $Z$ inside $\wti{X}$ has finitely many irreducible components.
In particular, if $Z$ is an absolutely irreducible subvariety of $X$,
then $Z$ is universally irreducible.
\end{cor}
\begin{proof}
The preimage $\wti{Z}$ of $Z$ inside $\wti{X}$
has finitely many irreducible components by Lemma \ref{lm:lift}(ii).
The conclusion follows from Proposition \ref{prop:finitely-many-irr} and
Corollary \ref{cor:promoting}.
\end{proof}

Now we focus on the case of divisors.
In order to ensure that, up to a finite \'etale cover, a divisor splits into the union
of absolutely irreducible divisors, we will consider projective varieties $\ol{X}$
with a subscheme $\pa X$ that satisfy the following property:
\[
\text{(IV)}\begin{cases}
\text{the codimension of $\pa X$ inside $\ol{X}$ is at least two (III$'_1$);}\\
\text{$\mathrm{Pic}(\ol{X})/\!\text{tors}=\ZZ\cdot \alpha$ with $\alpha$ ample;}\\
\text{the Picard number of $X'$ is $1$ and the divisibility
of $\alpha$ in $\mathrm{Pic}(X')/\!\text{tors}$}\\
\text{is uniformly bounded from above,
for every finite \'etale cover $X'\rar X$,}
\end{cases}
\]
where $X=\ol{X}\setminus\pa X$.
%
%
%\begin{itemize}
%\item[(IV)]
%the codimension of $\pa X$ inside $\ol{X}$ is at least two (III$'_1$);
%$\mathrm{Pic}(\ol{X})/\!\text{tors}=\ZZ\cdot \alpha$ with $\alpha$ ample;
%the Picard number of $X'$ is $1$ and the divisibility
%of $\alpha$ in $\mathrm{Pic}(X')/\!\text{tors}$ is uniformly bounded from above,
%for every finite \'etale cover $X'\rar X=\ol{X}\setminus\pa X$.
%\end{itemize}
%
The following is essentially due to Freitag (see \cite[Theorem 4.7]{FR} for the case
of a finite \'etale cover $X$ of $\calA_g$).

\begin{prop}[Existence of absolutely irreducible divisors]\label{prop:prime}
Let $\ol{X}$ be a projective variety with a subscheme $\pa X$
that satisfy properties (I) and (IV).
%
%\begin{itemize}
%\item[(III')]
%$H^2(\ol{X};\ZZ)/\mathrm{tors}=\ZZ\cdot\lambda$.
%\item[(III'')]
%$H^2(X';\QQ)=\QQ\cdot\lambda$ and the divisibility of $\lambda$ in $H^2(X';\ZZ)/\mathrm{tors}$
%is uniformly bounded from above for every finite \'etale cover $X'\rar X$.
%\end{itemize}
%Suppose moreover that
%%$X$ be a smooth variety with a projective compactification $\ol{X}$, whose boundary
%%$\pa X:=\ol{X}\setminus X$ has codimension greater than one, and such
%that $H^2(\ol{X};\ZZ)/\mathrm{tors}=\ZZ\cdot\lambda$.
%Suppose moreover that $H^2(X';\QQ)=\QQ\cdot\lambda$ and the divisibility of $\lambda$ in $H^2(X';\ZZ)/\mathrm{tors}$
%is uniformly bounded from above for every finite \'etale cover $X'\rar X$.
Then, for every effective Cartier divisor $\ol{D}\subset\ol{X}$,
there exists a finite \'etale cover $X'\rar X$ such that the preimage $D'\subset X'$
of $D\subset X$ is the union of finitely many absolutely irreducible divisors.
\end{prop}
\begin{proof}
By property (IV) there an integer $d>0$ such that $[\ol{D}]=d\cdot \alpha$.
Let now $X'\rar X$ be any finite \'etale cover and let $\ol{X}'$ be the normalization of $\ol{X}$ in the function field of $X'$.
Since $\ol{X}'$ is projective and $\pa X'$ has codimension at least two in $\ol{X}'$, property (IV) ensures that
$\mathrm{Pic}(\ol{X}')/\!\text{tors}\subseteq\ZZ\cdot(\alpha/d_0)$ for some integer $d_0\geq 1$.
Thus, each effective Cartier divisor in $\ol{X}'$ must be a positive multiple of $\alpha/d_0$. In particular, the divisor $D'$ obtained by pulling back $D$ via $X'\rar X$ has class $d\cdot\alpha$, and so
it can have at most $d\cdot d_0$ irreducible components.
It is then enough to pick $X'$ to be a finite \'etale cover of $X$ on which the number of irreducible components of $D'$ is maximal.
\end{proof}

\subsection*{The case of orbispaces}%\label{sec:orbispaces}

%A version of LHT as in Section \ref{sec:LHT}
%can be also phrased for orbifolds that are global quotients.
%Let $\ol{X}$ be an analytic space and let $\pa X$ and $\ol{Z}$ be subspaces of $\ol{X}$.
%Suppose moreover that a finite group $G$ acts on $\ol{X}$, preserving
%$\pa X$ and $\ol{Z}$, and so preserving $Z=\ol{Z}\setminus\pa X$.
%
%\begin{cor}[Improved LHT for c.i.~inside global quotients]\label{cor:LHT-orbifold}
%Suppose that $\ol{X}$, $\pa X$ and $\ol{Z}$ satisfy properties (I)-(III$^{ci}$)-(III$^{ci}_h$)
%in Section \ref{sec:LHT} with $h=2,3$.
%Then 
%\begin{itemize}
%\item[(i)]
%the locus $[Z/G]$ inside the orbifold $[X/G]$ is connected
%and $\pi_1([Z/G])\rar\pi_1([X/G])$ is an isomorphism if $h=3$ (resp. is surjective, if $h=2$);
%\item[(ii)]
%the preimage of $[Z/G]$ via an unramified cover of $[X/G]$ is connected.
%\end{itemize}
%\end{cor}
%\begin{proof}
%Claim (i) follows from Corollary \ref{cor:LHT} and claim (ii) from Corollary \ref{cor:connected-LHT}(ii).
%\end{proof}

The definition of absolutely irreducible and universally irreducible divisors
can be verbatim understood in the orbispace setting.
Thus, if $(\ol{X},\pa X)$ satisfy (I) and (IV), then
the conclusion of Proposition \ref{prop:prime} holds for every effective
Cartier divisor in $[\ol{X}/G]$.
The assertion is an immediate consequence of Proposition \ref{prop:prime}
and Corollary \ref{cor:promoting}.
As further examples, Proposition \ref{prop:finitely-many-irr}, Corollary \ref{cor:promoting} and Corollary \ref{cor:total}
verbatim hold in the setting of orbispaces.

\section{Subvarieties of $\calA_g(\Gamma)$ }\label{sec:Ag}

\subsection{Siegel upper half-space, congruence subgroups and multipliers}\label{ssc:siegel}
We denote by $\HH_g=\{\tau\in\op{Mat}_{g\times g}(\CC)\mid \tau=\tau^t, \op{Im}\tau>0\}$ the {\it{Siegel upper half-space}} of symmetric matrices with positive definitive imaginary part. It is a homogeneous space for the action of $\Sp_{2g}(\RR)$, where an element
$$
 \gamma=\left(\begin{matrix} A & B  \\ C &  D\end{matrix}\right) \in \Sp_{2g}(\RR)
$$
acts via
$$
\gamma\cdot \tau:=(A\tau+B)(C\tau+D)^{-1}.
$$
Such action can be lifted to $\widetilde{\mathcal{L}}:=\HH_g\times \CC$ as
\[
\gamma\cdot (\tau,u):=\left((A\tau+B)(C\tau+D)^{-1} , \mathrm{det}(C\tau+D)\cdot u\right),
\]
so that the projection $\widetilde{\mathcal{L}}\rar\HH_g$ onto the second factor
descends to a holomorphic line bundle on $\mathcal{L}\rar \calA_g(\Gamma)$ for every subgroup $\Gamma$ of $\Gamma_g$.

%We  denote  by  $\Gamma_g:=\Sp_{2g}(\ZZ)$  the {\it{Siegel modular group}}, and let $\Gamma_g(n):=\{\sigma\in\Gamma_g:\sigma\equiv I_{2g}\mod  n\}$ denote the 
%{\it{$n$-th principal congruence subgroup}} of $\Gamma_g$. The quotient $\calA_g=\HH_g/\Gamma_g$ is the moduli space of complex principally polarized abelian varieties (ppav), and $\calA_g(n)=\HH_g/\Gamma_g(n)$ is the moduli space of ppav with a choice of a full symplectic level~$n$ structure.
%We denote by $p:\HH_g\to\calA_g$ and $p_n:\HH_g\to\calA_g(n)$ the quotient maps,
%so that $\HH_g$ identifies to the universal cover of $\calA_g$
%(in the orbifold sense). 
%We recall incidentally that, for $n\geq 3$,
%the moduli space $\calA_g(n)$ is a smooth variety and $p_n$ is its universal cover in the standard sense.

\begin{nota}
Given a subvariety $\calY\subset\calA_g$, we will write $\wti{\calY}$ to denote
its preimage inside $\HH_g$,
$\calY(\Gamma)$ to denote its preimage in $\calA_g(\Gamma)$
and briefly
$\calY(n)$ for its preimage in $\calA_g(n)$.
\end{nota}

We assume $g\geq 2$ and $k$ half-integral. 
Consider now a subgroup $\Gamma$ of finite index of $\Gamma_g$.

 \begin{dfx}
A {\em{multiplier of weight $k$}} is  a map
 $\chi: \Gamma \to \CC^*$ such that $G(\gamma,\tau):=\chi(\gamma)\det(C\tau+D)^k$ satisfies the cocycle condition 
\[
G(\gamma'\gamma,\tau)=G(\gamma',\gamma\cdot\tau)G(\gamma,\tau)
\]
for all $\tau\in\HH_g$ and all $\gamma',\gamma\in\Gamma$.
A function $F:\HH_g\to\CC$ is called a {\em modular form of weight $k$ and multiplier $\chi$}
with respect to  $\Gamma$ if
 $$
  F(\gamma\circ\tau)=\chi(\gamma)\det(C\tau+D)^kF(\tau),\quad \forall \gamma
  \in\Gamma,\ \forall \tau\in\HH_g.
$$
\end{dfx}

Note that, for every $\gamma$ the function $\tau\mapsto \det(C\tau+D)$
has a holomorphic square root, which we denote by $\sqrt{\det(C\tau+D)}$.
When $k$ is half-integral, the expression $\det(C\tau+D)^k$
must be interpreted as $\sqrt{\det(C\tau+D)}^{2k}$.
When $k$ is integral, no square root is needed and the multiplier can be chosen to be just a character. We refer to \cite[Section I.3]{FR} for more details on modular forms and multipliers.

Denote by $A(\Gamma)$ the $\CC$-algebra of all modular forms (possibly with multiplier) with respect to $\Gamma$, which is finitely generated
(see \cite[Theorem 6.11]{F83}).

\subsection{Some topological properties of $\calA_g$}\label{ssc:topological-Ag}

%Let $\calA'_g=\HH_g/\Gamma'$ 
%for a certain subgroup $\Gamma'\subset\Gamma_g$ of finite index. 
%{\bf{
%Since $\Gamma_g$ acts properly on $\HH_g$,
%such $\calA'_g$ has a natural structure of
%complex analytic orbifold. Moreover the action of 
%$\Gamma_g(n)$ is free for $n\geq 3$, and so $\calA'_g$ is naturally a complex manifold
%if $\Gamma'\subseteq\Gamma_g(n)$ for some $n\geq 3$.
%
%Morever, we can regard the map $\calA'_g\rar\calA_g:=\HH_g/\Gamma_g$
%as a finite \'etale cover of orbifolds, or as a
%finite, possibly ramified, cover of complex varieties.

We recall that $\Gamma_g$ is finitely generated 
(see \cite[Hilfssatz 2.1]{mennicke}) and
so $\Gamma$ is. This in particular implies that the cohomology
of $\calA_g(\Gamma)$ is finitely generated. 

The {\it{Satake compactification}} $\ol{\calA}_g(\Gamma):=\mathrm{Proj}(A(\Gamma))$
of the coarse space of $\ol{\calA}_g(\Gamma)$ is a normal projective variety
(which may be singular for $g\geq 2$)
and so the coarse space of $\calA_g(\Gamma)$ can be regarded as a normal
quasi-projective variety. 
Moreover its boundary $\pa\calA_g(\Gamma):=\ol{\calA}_g(\Gamma)\setminus\calA_g(\Gamma)$ has pure codimension $g$.
%, and so
%$(\ol{\calA}_g(\Gamma),\pa\calA_g(\Gamma))$ satisfies property (I).
%
The boundary of $\ol{\calA}_g$ has a natural stratification, whose locally closed strata can be identified to $\calA_k$ for $k=0,\dots,g-1$ (see \cite[Chapter V.2]{faltings-chai}, for instance).

Now, assume $g\geq 2$ so that
$\Gamma$ contains a principal congruence
subgroup $\Gamma_g(n)$ by \cite{bms}, and so
$\calA_g$ is covered by some $\calA_g(n)$. Let us fix such an $n\geq 3$.

We recall the following:
\begin{itemize}
\item[(a)]
$\calA_g(n)$ is smooth and connected
(since $\HH_g$ is smooth and connected and $\Gamma_g(n)$ acts freely).
\item[(b)]
Since $\ol{\calA}_g(n)$ is normal
and its boundary has codimension at least two,
the (algebraic or analytic)
Picard groups of $\ol{\calA}_g(n)$ and $\calA_g(n)$ coincide.
\item[(c)]
The algebraic and analytic Picard groups of $\ol{\calA}_g(n)$ coincide (use \cite{serre}), and the same holds for $\calA_g(n)$.
%Moreover,
%the first and second (co)homology groups of 
%$\calA_g(n)$ and $\ol{\calA}_g(n)$ coincide.
\item[(d)]
$H^1(\calA_g(n);\QQ)=0$ (see \cite{kazdan}) and so $H_1(\calA_g(n);\ZZ)$ is a finite group; $H^2(\calA_g(n);\QQ)$ has dimension $1$ for $g\geq 3$ 
(see \cite{borel1} and \cite{borel2}) and it is generated by
the ample class $\lambda=c_1(\mathcal{L}(n))$.
\item[(e)]
For $g\geq 3$ the natural map $\Pic(\calA_g(n))\rar H^2(\calA_g(n);\ZZ)$ is an isomorphism
(as in \cite{Fr}).
%, just use the above considerations
%and the connecting homomorphism
%$\Pic(\calA_g(n))=H^1(\calA_g(n);\mathcal{O}^*)\rar H^2(\calA_g(n);\ZZ)$ induced by the exponential sequence
%$0\rar \ZZ\rar \mathcal{O}\rar \mathcal{O}^*\rar 1$).
\item[(f)]
Each effective divisor in $\calA_g(n)$ is the zero set
of a modular form of half-integral weight $a/2$, possibly with a multiplier (cf.~\cite{Fr}, \cite{deligne} and \cite{FR}). Its class in $H^2(\calA_g(n);\ZZ)/\mathrm{tors}$
is $a\lambda/2$.
\item[(f')]
$H^2(\calA_g(n);\ZZ)/\mathrm{tors}$ is generated by $\lambda$ or by $\lambda/2$.
\item[(g)]
Every effective
divisor in $\ol{\calA}_g(n)$ is the closure of a divisor in $\calA_g(n)$ and it is ample.
Moreover its boundary has codimension $g$ or $g-1$ inside it.
\end{itemize}

Note that (f') is a consequence of (f)
and that (g) depends on the fact that $\pa\calA_g(n)$ has codimension at least two.

Together with the above discussion, (a) and (f)
imply that $(\calA_g(n),\pa\calA_g(n))$ satisfied 
property (I) introduced in Section \ref{sec:LHT}
and property (IV) introduced in Section \ref{sec:absol}
for all $g\geq 3$ and $n\geq 3$.

Suppose now that $\Gamma$ is a subgroup of $\Gamma_g$
of finite index. Then all properties (b-g) still hold
for $\calA_g(\Gamma)$. Moreover, if $\Gamma$
is contained in a $\Gamma_g(n)$
for some $n\geq 3$, then property (a) holds too.

\subsection{Irreducible modular forms}

We begin with a consequence of the Lefschetz hyperplane theorem
as  stated  in Theorem \ref{thm:divisor}. 

\begin{thm}[Connectedness of zero loci of modular forms]\label{thm:connectedness}
Let $g\geq 3$ and let $\calA_g(\Gamma)$ for some finite-index subgroup $\Gamma$ of $\Gamma_g$. Then every divisor $D$ in $\calA_g(\Gamma)$ pulls back to a connected divisor $\wti{D}$ in $\HH_g$.
Moreover, if $D$ is locally irreducible, then $\wti{D}$ is irreducible.\\
For $g=2$ the same conclusions hold if 
$\partial{D}$ consists of finitely many points.
\end{thm}
\begin{proof}
Up to replacing $\Gamma$ by $\Gamma_g(n)\subset\Gamma$ for a suitable $n\geq 3$, and $D$ by its pull-back
to $\calA_g(n)$, we can assume that $\calA_g(\Gamma)$ is a smooth variety
and we let $\ol{\calA}_g(\Gamma)$ be its Satake compactification with $\pa\calA_g(\Gamma)=\ol{\calA}_g(\Gamma)\setminus\calA_g(\Gamma)$.
By what we have recalled in Section \ref{ssc:topological-Ag}, $\ol{D}$ is an ample Cartier divisor in $\ol{\calA}_g(\Gamma)$
and so, for $g\geq 3$, the triple $(\ol{\calA}_g(\Gamma),\pa\calA_g(\Gamma),\ol{D})$ satisfy properties
(I)-(II)-(III$'_2$)
introduced in Section \ref{sec:LHT}.
By Theorem \ref{thm:divisor} it follows that $\pi_1(D)\rar\Gamma$ is surjective,
and so $\wti{D}$ is connected by Lemma \ref{lm:lift}(iii).
The second claim is a consequence of Proposition \ref{prop:irred}(i).\\
For $g=2$ it is enough to note that the required condition ensures that $\ol{D}$
satisfies (III$_2$), so that Theorem \ref{thm:divisor} still applies.
\end{proof}

Note that, for $g\geq 4$, the same argument as above
shows that $\pi_1(D)\cong\Gamma$ and so its lift $\wti{D}$ is simply connected.

\begin{example}[Eisenstein series]
For every $g$  we write  $\Gamma_{g,0}\subset \Gamma_{g}$ for the subgroup defined by $C=0$.
For every integer $k>\frac{g+1}{2}$,  the Eisenstein series
\[
E_{2k}(\tau):=\sum_{\gamma\in\Gamma_g/\Gamma_{g,0}} \det(C\tau+D)^{-2k},\qquad
\text{where}\quad\gamma=
\left(\begin{array}{cc}
A & B\\
C & D
\end{array}\right)
\]
is a modular form  of weight $2k$ and determines a divisor $\ol{D}$ in $\ol{\calA}_g$ that does not contain $\pa\calA_g$. It follows that $\wti{D}\subset\HH_g$ is connected for $g\geq 2$.
\end{example}

In some cases the generic element $\ol{E}$ in the linear system of a divisor
$\ol{D}\subset \ol{\calA}_g(\Gamma)$ has smooth internal part $E=\ol{E}\cap\calA_g(\Gamma)$: for example, this occurs by Bertini if
$|\ol{D}|$ is very ample away from the boundary. One can thus
apply Theorem \ref{thm:connectedness}
to conclude that the preimage of $E$ in $\HH_g$ is irreducible.

Since the Satake compactification of $\calA_g(\Gamma)$ is  
${\rm Proj} (A(\Gamma))$, the linear system of modular forms with respect to $\Gamma$ of sufficiently divisible weight is certainly very ample on $\ol{\calA}_g(\Gamma)$.
On the other hand, 
given an arbitrary half-integral weight and a system of multipliers, there exists a 
cofinal subset $\{\Gamma_g[n,2n]\}_n$ of finite-index subgroups of $\Gamma_g$
such that
the linear system of modular forms with the given weight and multipliers is very ample on every $\calA_g(\Gamma_g[n,2n])$ for $n\geq 3$
(see \cite[Theorem 10.14]{mumford:tata3} and \cite{SM96}).
As consequence, we have that for high enough level (in a cofinal set), modular forms of weight $1/2$ give an embedding of the modular variety.\\
%And we recall that for $g\geq 2$ the minimal weight is exactly $1/2$, {\cite[Theorem 4.6]{FR}}.\\

Our aim now is to investigate the irreducibility of the preimage of
{\it{any}} divisor of $\calA_g(\Gamma)$.
In order to do that, we need to recall the following 
result from {\cite[Theorem 4.7]{FR}}, which can be also obtained as a consequence of Proposition \ref{prop:prime}.

%\begin{terminology}
% {\bf A modular form $f$ with respect to the subgroup $\Gamma'$
% of finite index of $\Gamma_g$ is a {\it{ prime form} } 
% if the associated divisor on $\calA'_g$ is irreducible.
% We call $f$ an {\it{absolutely prime form}} if it is a prime form with respect to arbitrary small congruence subgroups. We call $f$   a {\it{ universally prime form  }}  if its associated divisor in $\HH_g$ is irreducible.}
%\end{terminology}

\begin{lm}[Factorization into absolutely irreducible forms]\label{lm:factorization}
For $g\geq 3$ the following hold.
\begin{itemize}
\item[(i)]  
Each  non-vanishing modular form of weight $1/2$ is absolutely irreducible.
 \item[(ii)] 
Let $F$ be a non-vanishing  modular form with respect to a finite-index subgroup $\Gamma$
of $\Gamma_g$. 
There exists a factorization of $F$ as
$$F= f_1^{r_1}\cdot\dots\cdot f_u^{r_u},$$
where $r_1,\dots,r_u$ are positive integers and $f_1,\dots,f_u$ are distinct, absolutely irreducible, modular forms.
Moreover, such factorization is unique up to reordering the $f_i$'s. 
\end{itemize}
\end{lm}

%As an immediate consequence of Theorem \ref{thm:prime} we have the following

%\begin{thm}[Absolutely prime modular forms are universally prime] \label{thm:ag}
%For $g\geq 2$,
%every absolutely irreducible divisor $D\subset \HH_g/\Gamma$
% is universally irreducible.
%\end{thm}

Note that Lemma \ref{lm:factorization} does not hold for $g=2$: an example
is discussed in Section \ref{sec:tn2}.

As a preparation for 
a deeper investigation regarding liftings of subvarieties of $\calA_g(\Gamma)$ to $\HH_g$,
we will first discuss the case of theta constants.

\subsection{Theta functions}\label{sec:theta}
The {\it{(first-order) theta function $\vartheta:\HH_g\times\CC^g\rar\CC$}}
is defined as
\[
\vartheta(\tau,z):=\sum\limits_{n\in\ZZ^g} \exp \pi i \left(
n^t\tau n+2n^t z\right)
\]
and it is even in $z$. 
%It descends to a section of a line bundle over $\calX_g$and 
%Its zero locus $\wti{\Theta}$ is the lift of a divisor
%$\Theta\subset\calX_g$, called the {\it{universal theta divisor}}.

For $\ep,\de\in (\ZZ/2)^g$
we define { \it the (first order) theta function 
$\t\chars\ep\de:\HH_g\times\CC^g\rar\CC$
with characteristic $\chars\ep\de$} as
$$
\tt\ep\de(\tau,z):=\sum\limits_{n\in\ZZ^g} \exp \pi i \left[\left(
n+\frac{\ep}{2}\right)^t\tau \left(n+\frac{\ep}{2}\right)+2\left(n+\frac{\ep}{2}\right)^t\left(z+
\frac{\de}{2}\right)\right]
$$
and note that $\t\chars{0}{0}=\t$.
%We will usually write $\ep,\de$ as rows (or sometimes columns, if notationally more convenient)  of $g$ zeroes and ones, and operate with them over $\ZZ_2$ unless stated otherwise. 
%The {\it{theta constant}} 
%$\theta:=\vartheta(\cdot,0):\HH_g\rar\CC$ associated to $\vartheta$
%is a modular form of weight $1/2$ relative  to a suitable subgroup of $\Gamma_g$.
%and it is descends to  a section of a line bundle over $\calA_g(2)$.\\
%and its zero locus is denoted by $\tn\chars0 0$.

The {\it{characteristic}} $\chars{\ep}{\de}\in(\ZZ/2)^{2g}$ is called {\it{even}} or {\it{odd}} depending on whether the scalar product $\ep\cdot\de$ is zero or one as an element of $\ZZ/2$. It turns out that there are $2^{g-1}(2^g+1)$ even characteristics and $2^{g-1}(2^g-1)$ odd ones.
As a function of~$z$, the theta function $\t\chars{\ep}{\de}(\tau,z)$ is even (resp.~odd) 
if its characteristic is even (resp.~odd). 

The {\it{theta constant $\theta\chars{\ep}{\de}:\HH_g\rar\CC$}},
already defined in Section \ref{sec:main},
is the value of theta function $\t\chars{\ep}{\de}$ at $z=0$.
%, namely
%the function $\theta_m:\HH_g\rar\CC$ defined as
% and theta gradients are the values of the $z$-gradient of the theta function, evaluated at $z=0$. We will drop the $z$ variable from notation in both cases, and write
%$$ \theta_m(\tau):=\t_m(\tau,0).$$
Note that theta constant $\theta\chars{\ep}{\de}$ vanish identically for $\chars{\ep}{\de}$ odd.

% while theta gradients vanish identically for $m$ even. 
Similarly to $\t=\t\chars{0}{0}$, the other even theta constants $\t\chars{\ep}{\de}$ are examples of modular forms of weight $1/2$ for the subgroup $\Gamma_g(2)$
and their zero locus in $\calA_g(2)$ is denoted by $\tn\chars\ep\de$. 
%In fact, these are seen to transform as follows (cf. \cite{Igbook} and \cite{FR} for the general formula):
%\[
%\tt\ep\de(\gamma \cdot \tau,0) = \left(\kappa(\gamma) \chi_{\ep,\de}(\gamma) \det{(C \tau + D)}^{\frac{1}{2}}\right)\cdot \tt\ep\de(\tau,0) \quad \quad \forall \gamma \in \Gamma_g(2)
%\]
%where $\kappa(\gamma)$ is an eighth root of the unity for any $\gamma=\left(\begin{array}{cc} A & B  \\ C &  D\end{array}\right)$ and $ \chi_{\ep,\de}$ is a character.
%In particular, $\theta\chars\ep\de$ is a modular form of weight $\frac{1}{2}$ and
%can be seen as a well-defined section on a line bundle
%over $\calA_g(2)$, whose zero locus in $\calA_g(2)$ is denoted by $\tn\chars\ep\de$.
%Similarly, $\tt\ep\de$ is a section of a line bundle over $\calX_g(2)$.

The group $\Gamma_g$ acts on characteristics, considered as elements of $(\ZZ/2)^{2g}$, via an affine action of its quotient $\Sp_{2g}(\ZZ/2)=\Gamma_g/\Gamma_g(2)$. In particular, this is to say that $\Gamma_g(2)$ is precisely equal to the subgroup of~$\Gamma_g$ that fixes each characteristic. Moreover, $\Gamma_g/\Gamma_g(2)$ acts transitively  on the subset
of even characteristics.
 We  consider  the union $\tn(2)$ of all $\tn\chars\ep\de\subset \calA_g(2)$, as $\chars\ep\de$ ranges among the $2^{g-1}(2^g+1)$ even characteristics. It is the pull-back of a divisor $\tn\subset \calA_g$. Thus $\tn$ is the image of a single
 $\tn\chars\ep\de$
via the cover $\calA_g(2)\rar\calA_g$ for any even $\chars\ep\de$.
%he locus where   $\prod_{\chars\ep\de\,\,even}\tc\ep\de(\tau,0)$ vanishes.

%%We recall from~\cite{Igbook,sm} that the orbits of $\Gamma_g$ on tuples of characteristics are fully characterized by parity of characteristics, by the a/syzygy properties of triples of characteristics, and by linear relations with an even number of terms.
%We will not use the details of this except to note that the zero loci of $\theta_m(\tau)$  defined as submanifolds of~$\HH_g$, are invariant under the action of~$\Gamma_g(2)$, and thus are preimages of well-defined subvarieties
%$$
% \tn\chars\ep\de:=\{\tc\ep\de(\tau)=0\}\subset\calA_g(2)
%$$
%$$\gn\chars\ep\de:=\{\grad\tc\ep\de(\tau)=0\}\subset\calA_g(2).$$
%For any $2\le k\le g$ we define $\tn^k\chars\ep\de\subset\tn\chars\ep\de$ to be the locus where the rank of the Hessian matrix $\left(\partial_{z_a}\partial_{z_b}\tc\ep\de(\tau,z)|_{z=0}\right)_{1\le a,b\le g}$ is at most $k$.
%We further denote $\tn\subset\calA_g$ the locus where at least one even theta constant $\tc\ep\de(\tau,0)$ vanishes
%
%
%
%Since the action of $\Gamma_g/\Gamma_g(2)$ permutes  even theta characteristics and the loci $\tn\chars\ep\de$ 
%%and respectively $\gn\chars\ep\de$  transitively, it follows that their images
%$$
% \tn:=p(\tn\chars\ep\de)\subset\calA_g
% %\quad{\rm and}\quad \gn:=p(\gn\chars\ep\de)\subset\calA_g
%$$
%are independent of the choices of even 
%%or odd~
%$\chars\ep\de$
%, respectively
Geometrically,~$\tn$ is the locus of ppav whose theta divisor has a singularity at an even two-torsion point. 
%The locus~$\tn$ was studied classically, and 

%in this study is the following.
% (necessarily of even multiplicity, at least 2)

 %, while~$\gn$ is the locus of ppav whose theta divisor has a singularity (necessarily of odd multiplicity, at least 3) at an odd two-torsion point on the abelian variety.
%The loci $\tn^k$ in $\calA_g$ are similarly independent of the choice of characteristics.

%For low genera these loci have a simple geometric interpretation, which is part of the motivation for studying them. Here are the geometric descriptions for low genera:

Thus we have 

\begin{prop}[Irreducible components of $\tn(2)$] \label{prop:tn-irred}
Let $g\geq 3$ and let $\chars\ep\de\in(\ZZ/2)^{2g}$ be an even characteristic. Then
\begin{itemize}
\item[(i)]
the divisor $\tn\chars\ep\de$ is irreducible, and so are its preimages
in $\calA_g(2n)$ for every $n>1$;
\item[(ii)]
the divisor $\tn(2n)$ is connected, not locally irreducible and consists of 
$2^{g-1}(2^g+1)$ irreducible components;
\item[(iii)]
the divisor $\tn\subset\calA_g$ is irreducible but not locally irreducible.
\end{itemize}
\end{prop}
\begin{proof}
(i) The divisor $\theta\chars\ep\de$ is the zero locus inside $\calA_g(2)$
of a modular form $\theta\chars\ep\de(\tau)$ of minimal weight $\frac{1}{2}$, which then 
is absolutely irreducible by Lemma \ref{lm:factorization}(i).

(ii) Since $g\geq 3$, every pair of irreducible components of $\tn(2)$
intersect. In fact, the closure of two distinct $\tn\chars{\ep}{\de}$ and $\tn\chars{\ep'}{\de'}$
inside the Satake compactification $\ol{\calA}_g(2)$ must intersect each other, because theta-nulls are ample.
Since $\partial\calA_g(2)$ inside $\ol{\calA}_g(2)$ has codimension $g\geq 3$,
it follows that  $\tn\chars{\ep}{\de}\cap\tn\chars{\ep'}{\de'}$ is not contained
inside $\partial\calA_g(2)$, and so they must intersect inside $\calA_g(2)$.
As a consequence, $\tn(2)$ is connected and not locally irreducible, and $\tn(2n)$ is  so too for all $n$.
Moreover, being each $\tn\chars\ep\de$ absolutely irreducible by (i), every irreducible component
of $\tn(2n)$ is a connected \'etale cover (in the orbifold sense) of some $\tn\chars\ep\de$.
Since there are exactly $2^{g-1}(2^g+1)$ even characteristics, $\tn(2n)$ has $2^{g-1}(2^g+1)$ irreducible components.

(iii) Since $\tn$ is the image of any even $\tn\chars\ep\de$, it follows from (i) that $\tn$ is irreducible.
Moreover, $\tn(2)\rar\tn$ is an \'etale cover (in the orbifold sense), and so $\tn$ cannot be locally irreducible by (ii).
\end{proof}

%As a consequence of our Theorem \ref{thm:ag},
 We then have the following
%As a first application, we analyze the modular form $\tc\ep\de$.
  
  \begin{cor}[$\tc\ep\de$ is universally irreducible for $g\geq 3$]\label{cor:tn}
  For $g\geq 3$ and for any even characteristics $\chars\ep\de$,
  the divisor $\wti{\tn}\chars\ep\de\subset\HH_g$ is normal and irreducible.
 \end{cor}
 \begin{proof}
We have seen that, since $\t\chars\ep\de$ is modular form of weight $1/2$, it is absolutely irreducible.
Moreover, $\tn\chars\ep\de$ is a divisor inside
  $\calA_g(2)$ and regular in codimension $1$ and so it is normal (see \cite{cvg}). Since normality is a local property, $\wti{\tn}\chars\ep\de$ is normal too.
 Hence $\wti{\tn}\chars\ep\de$ is irreducible by Theorem \ref{thm:connectedness}.
  %of $\wti{\tn}\chars\ep\de$ follows from Theorem \ref{thm:ag}.
  \end{proof}

%For  a deeper investigation it would be useful to
%recall the known characterizations of the locus of hyperelliptic Jacobians
%in $\calA_g(2)$ in terms of vanishing of certain sets of theta constants, and we describe  the %irreducible components of the corresponding loci on the level covers.  
 
%\subsection{ Locus of Jacobians and of hyperelliptic Jacobians}
\subsection{Lifting subvarieties that contain $\calJ_g$ or $\calH_g$}\label{sec:tn}
 
We use $\calJ_g\subset\calA_g$ to denote the locus of Jacobians of smooth genus $g$ curves, and denote by $\calH_g\subset\calJ_g$  the locus of hyperelliptic Jacobians. 
We recall that $\calJ_g$ is the image of the moduli space of curves $\calM_g$
via the Torelli morphism $t_g:\calM_g\rar\calA_g$
and that $\calH_g$ is the image of the hyperelliptic locus  $\calHM_g\subset\calM_g$
 via $t_g$.
Moreover, {\it{as orbifolds}} $\calJ_g\setminus\calH_g$ and $\calH_g$ are smooth and
$\calJ_g$ has unibranched singular locus $\calH_g$
(for example, $\calJ_g(n)\setminus\calH_g(n)$ and $\calH_g(n)$ are smooth varieties
for all $n\geq 3$).

We recall the following fact.

\begin{prop}[Irreducible components of $\calH_g(2)$ \cite{tsu}]\label{prop:irr-H2}
For $g\geq 3$, the hyperelliptic locus $\calH_g(2)$ has 
$$ 2^{g^2} \prod_{k=1}^{g}( 2^{2k}-1)/(2g+2)!  = |\Sp_{2g}(\ZZ/2)|/ |\SY_{2g+2}|$$ irreducible components.
Moreover,
the preimage in $\calA_g(2n)$ of an irreducible component of $\calH_g(2)$
is irreducible.
\end{prop}

Since $\calH_g$ is irreducible, the group 
$\Gamma_g$ transitively acts on the set of irreducible components of $\calH_g(2)$.
Hence the above
Proposition \ref{prop:irr-H2} is a consequence of the following, cf.~\cite[Lemma 8.12]{mutheta2}, or \cite[Theorem 1]{ac}.

\begin{prop}[Fundamental group of the hyperelliptic locus]\label{prop:pi1-hyp}
If $\iota:\calH_g\rar\calA_g$ is the inclusion of the locus of hyperelliptic Jacobians
with $g\geq 2$, 
then the image of $\pi_1(\iota):\pi_1(\calH_g)\rar\pi_1(\calA_g)=\Gamma_g$ fits into the following exact sequence
\[
1\rar \Gamma_g(2)\rar \mathrm{Im}(\pi_1(\iota))\rar \SY_{2g+2}\rar 1.
\]
In particular, any irreducible component $H$ of $\calH_g(2n)$
satisfies $$\pi_1(H)\thra \Gamma_g(2n).$$
\end{prop}

In fact, using Proposition \ref{prop:pi1-hyp} 
and Lemma \ref{lm:lift}(iii) we can draw the following conclusion.

\begin{cor}[Irreducible components of $\wti{\calH}_g$]
The preimage in $\HH_g$ of a connected component of $\calH_g(2)$
is smooth and connected.\\
The   preimage of $\calH_g$ in $\HH_g$ consists of  
$2^{g^2} \prod_{k=1}^{g}( 2^{2k}-1)/(2g+2)!$ smooth irreducible components.
\end{cor}

 Similarly, for the Jacobian locus we have:

\begin{lm}[Irreducibility of the Jacobian locus $\wti{\calJ}_g$]\label{lemma:irr-Jacobian}
The Jacobian loci $\calJ_g(\Gamma)$ are irreducible
for all subgroups $\Gamma$ of $\Gamma_g$.
%The preimage $\wti{\calJ}_g$ in $\HH_g$ is irreducible, and so 
%$\calJ_g(n)$ for all $n$.
\end{lm}
\begin{proof}
It is well-known that $\calM_g$ is a connected orbifold.
Since the homomorphism
$\pi_1(\calM_g)\rar\pi_1(\calA_g)$ of orbifold fundamental
groups induced by $t_g$ is surjective (see Fact \ref{fact:rho-surjective}), 
the irreducibility of the Jacobian locus $\wti{\calJ}_g$ in $\HH_g$
follows from Lemma \ref{lm:lift}(iii). 
We conclude, since $\calJ_g(\Gamma)=\wti{\calJ}_g/\Gamma$.
%
%
%Since $\calM_g (n)$ is smooth and connected,
%and so $\calM_g$ is a connected orbifold,
%by Lemma \ref{lm:lift}(iii) it is enough to show
%that $\pi_1(t):\pi_1(\calM_g)\rar\pi_1(\calA_g)$ is surjective
%as a map of orbifold fundamental groups.
%Such homomorphism can be identified to the symplectic representation
%$\MCG_g\rar \Sp_{2g}(\ZZ)$, which is classically known to be surjective
%(see, for instance, Chapter 6 of \cite{farb-margalit}).
\end{proof}

%\subsection{Lifting subvarieties that contain $\calJ_g$ or $\calH_g$}\label{sec:tn}

Here we specialize some of the results obtained in Section \ref{sec:topological} to subvarieties of $\calA_g(\Gamma)$ that contain the Jacobian
or the hyperelliptic locus.\smallskip

 %The first statement we are interested in is the following.

\begin{prop}[Subvarieties containing hyperelliptic or Jacobian locus]\label{prop:irr-Mg}
%[Subvarieties containing the Jacobian locus]
%Let $\calA'_g$ be a finite level that dominates $\calA_g(n)$ and
Let $Z\subset\calA_g(\Gamma)$ be an irreducible subvariety.
\begin{itemize}
\item[(i)]
If a Zariski-open subset of $\calJ(\Gamma)$ is contained in $Z$, then 
$\wti{Z}\subset\HH_g$ is connected, with finitely many irreducible components.
Moreover, if such Zariski-open subset of $\calJ(\Gamma)$
is not entirely contained inside the singular locus of $Z$,
then $\wti{Z}$ is irreducible.
\item[(ii)]
Assume that $\calA_g(\Gamma)$ dominates $\calA_g(2)$.
If a Zariski-open subset of $\calH_g(\Gamma)$ is contained in $Z$, then 
$\wti{Z}\subset\HH_g$ is connected, with finitely many irreducible components.
Moreover, if such Zariski-open subset of $\calH_g(\Gamma)$
is not entirely contained inside the singular locus of $Z$,
then $\wti{Z}$ is irreducible.
\end{itemize}
\end{prop}
\begin{proof}

(i) The inclusion $\calJ_g(\Gamma)\hra\calA_g(\Gamma)$ induces
a surjection at the level of orbifold fundemental groups
by Corollary \ref{cor:rho-surjective}.
The conclusion follows from Corollary
\ref{cor:connected-LHT} for the connectedness of $\wti{Z}$ and Corollary
\ref{cor:irred-crit} for the irreducibility of $\wti{Z}$.

(ii) Since $\Gamma$ is contained inside $\Gamma_g(2)$,
the inclusion $H\hra\calA_g(\Gamma)$ of
a connected component of $\calH_g(\Gamma)$ induces
a surjection at the level of orbifold fundemental groups
by Proposition \ref{prop:pi1-hyp}.
Again, we conclude by Corollary Corollary
\ref{cor:connected-LHT} and Corollary \ref{cor:irred-crit}.
\end{proof}

As a first example, one can apply the above Proposition 
\ref{prop:irr-Mg}(ii) again to the case of $Z=\tn\chars\ep\de$
inside $\calA_g(2)$ to deduce the irreducibility of $\wti{\tn}\chars\ep\de$. In fact in Theorem 9.1 at page 137 of \cite{mutheta2},
Mumford gives  a characterization of a component $\wti{H}$
of $\wti{\calH}_g(2)$ in terms of  vanishing of theta constants. Moreover, in \cite{sm} it is proved that these equations cut $\wti{H}$ smoothly, and so $\wti{H}$ is contained in the smooth locus of $\tn\chars\ep\de$.\\ %
% 
%\begin{cor}[Irreducibility of $\wti{\tn}\chars\ep\de$]\label{cor:tn}
%For every even characteristics $\chars\ep\de$, the lift
%$\wti{\tn}\chars\ep\de\subset\HH_g$ is irreducible.
% \end{cor}
%\begin{proof}
%The hyperelliptic locus $\calH_g$ is contained inside $\tn$.
%Moreover, for every irreducible component $\tn\chars\ep\de$ of $\tn(2)$
%there exists an irreducible component $\calH_g(2)_{\irr}$ contained in $\tn\chars\ep\de$
%that does not sit entirely inside the singular locus of $\tn\chars\ep\de$.
%
%{\bf{INSERIRE RIFERIMENTO???}}
%
%The conclusion then follows from Proposition \ref{prop:irr-Hg}(ii).
%\end{proof} 
% 

As a second example, we mention the irreducible component 
of $\calN_k$ that contains the Jacobian locus. Assume $g\geq 4$.
It was proven by Andreotti-Mayer \cite{anma} that the theta divisor of a Jacobian of dimension $g$ has singular locus of dimension at least $g-4$.
If $\calN_k\subset\calA_g$ is the locus of Abelian varieties whose theta divisor has singular locus of dimension at least $k$,
%(see Section \ref{sec:vertical}), 
then
for all $k\leq g-4$ there exists an irreducible component
$\calN^J_k$ of $\calN_k$ that contains the Jacobian locus.
Then $\wti{\calN}^J_k$ is connected with finitely many irreducible components
by Proposition \ref{prop:irr-Mg}(i).
%In a successive investigation  we will check that $\calJ_g$ sits
%inside the singular locus of  $\calN^J_k$ for all $k\leq g-5$.  
%Hence we cannot immediately conclude that $\wti{\calN}^J_k\subset\HH_g$ is irreducible.

%%%%%%%%%%%%%%%%%%%%%%%%%%%%%%%%%%%%%%%%%%%%%%%%%%%%%%%%%%%%%%%%%%%%%%%%%%%%%%%%%%%%%%5
\section{Universally irreducible subvarieties of $\calA_g(\Gamma)$}\label{sec:universal}

%%%%%%%%%%%%%%%%%%%%%%%%%%%%%%%%%%%%%%%%%%%%%%%%%
\subsection{Rational translates}\label{sec:translates}

We recall that the Siegel upper half-space $\HH_g$ is transitively acted on
by $\Sp_{2g}(\RR)$. 
Given a finite-index subgroup $\Gamma$ of $\Gamma_g$,
the action by a rational element $M\in\Sp_{2g}(\QQ)$ on $\HH_g$
induces a diagram
\[
\xymatrix{
& \calA_g(\Gamma^M) \ar[dl]^{p} \ar[dr]_{q_M}^{\sim} &\\
\calA_g(\Gamma) && \calA_g(\Gamma_M)
}
\] 
for suitable finite-index subgroups $\Gamma^M$ and $\Gamma_M$ that only depend on $\Gamma$ and $M$. The multiplication by $M$, that sends a subset of $\HH_g$ to its $M$-translate in $\HH_g$,
descends to $q_M\circ p^{-1}$, which sends a subset of $\calA_g(\Gamma)$ to a subset of $\calA_g(\Gamma_M)$.
In this section we are describing the properties of such construction.\\

Let $M\in\Sp_{2g}(\QQ)$. We denote by $c_M$ the conjugation by $M$
defined as $c_M(\gamma):=M\gamma M^{-1}$.
Given a finite-index subgroup $\Gamma$ of $\Gamma_g$,
the conjugate $c_M(\Gamma)$ is a subgroup of $\Sp_{2g}(\QQ)$.
We define $\Gamma_M:=\Gamma_g\cap c_M(\Gamma)$
and $\Gamma^M:=c_{M^{-1}}(\Gamma_M)$,
which are subgroups of $\Gamma_g$.

\begin{example}
If $M=\left(\begin{array}{cc}d\cdot I_g & 0 \\ 0 & d^{-1}\cdot I_g\end{array}\right)$ with $d\geq 2$ integer and $\Gamma=\Gamma_g$,
then
\[
\Gamma_M=
\left\{\left(\begin{array}{cc}A & B \\
C & D\end{array}\right)\in\Gamma_g\ \Big|\ C\equiv 0\pmod{d^2}   \right\}.
\]
\end{example}

%be the subgroup
%of $\gamma\in\Gamma_g$ such that $M^{-1}\gamma M\in\Gamma'$.
%We also let $\Gamma'(M,M^{-1}):=\Gamma'(M)\cap\Gamma'(M^{-1})$.

\begin{lm}[$\Gamma_M$ and $\Gamma^M$ have finite index in $\Gamma_g$]\label{lemma:matrixM}
Let $d\geq 1$ be an integer and $M\in\Sp_{2g}(\QQ)$ such that $dM$ has integral entries.
\begin{itemize}
\item[(i)]
If $\Gamma\supseteq\Gamma_g(\ell)$, then $\Gamma_M\supseteq\Gamma_g(\ell d^2)$.
\item[(ii)]
$\Gamma_M$ has finite index in $\Gamma_g$.
\item[(iii)]
$\Gamma^M$ is a finite-index subgroup of $\Gamma$.
\end{itemize}
\end{lm}
\begin{proof}
Note first that both $d\cdot M$ and $d\cdot M^{-1}$ have integral entries.
This is immediate, since
$M=\left(\begin{array}{cc} A & B\\  C& D\end{array}\right)$ is symplectic
and so
$M^{-1}=\left(\begin{array}{cc} D^t & -B^t\\  -C^t & A^t\end{array}\right)$.

(i)
Let now $\gamma$ be any element of $\Gamma_g(\ell d^2)$. We want to show that $\gamma\in\Gamma_M$.
Since $\Gamma_g(\ell d^2)\subseteq\Gamma_g(\ell)\subseteq\Gamma$, it is enough
to show that $\gamma$ belongs to $c_M(\Gamma_g(\ell))\subseteq c_M(\Gamma)$.
% and 
%to $M^{-1}\Gamma_g(\ell) M\subseteq M^{-1}\Gamma' M$.

Since $\gamma\in\Gamma_g(\ell d^2)$, we can write $\gamma=I+\ell d^2N$ for some integral $N$. Then
\[
M^{-1}\gamma M=M^{-1}(I+nd^2N)M=I+\ell(dM^{-1})N(dM).
%,
%\qquad
%M\gamma M^{-1}=M(I+nd^2N)M^{-1}=I+\ell(dM)N(dM^{-1}).
\]
Hence $M^{-1}\gamma M$ belongs to $\Gamma_g(\ell)$ and so
$\gamma\in M\Gamma_g(\ell)M^{-1}$, %and to $M^{-1}\Gamma_g(\ell)M$, 
as desired.

(ii)
Since $\Gamma$ has finite-index in $\Gamma_g$, it
contains a principal congruence subgroup $\Gamma(\ell)$ for some $\ell\geq 1$.
Hence, $\Gamma_M$ contains $\Gamma(\ell d^2)$ and so it has finite index
in $\Gamma_g$.

(iii) follows from (ii), since $\Gamma^M=c_{M^{-1}}(\Gamma_M)=\Gamma\cap c_{M^{-1}}(\Gamma_g)$.
\end{proof}

%Consider now $^M\Gamma:=c_{M^{-1}}(\Gamma_g)\cap \Gamma$, which is
%a finite-index subgroup of $\Gamma$ by Lemma \ref{lemma:matrixM}.
By Lemma \ref{lemma:matrixM}, we have the following diagram of finite-index subgroups of $\Gamma_g$
\[
\xymatrix{
& \Gamma^M \ar@{_(->}[dl] \ar[rd]^{\sim}_{c_M} &\\
\Gamma && \Gamma_M
}
\]
%Call $\Gamma'=\Gamma(M)$ and $\Gamma''=\Gamma(M,M^{-1})$.
The isomorphism of $\HH_g$ given by the multiplication by $M$
descends to an isomorphism $q_M$ as in the following diagram
\[
\xymatrix{
& \calA_g(\Gamma^M) \ar[dl]^p \ar[dr]^\sim_{q_M} & \HH_g\ar[l] \ar[dr]_{M}^\sim \\
\calA_g(\Gamma) && \calA_g(\Gamma_M) & \HH_g \ar[l] 
}
\]

\begin{dfx}[Rational translate]
Given $Z\subset\calA_g(\Gamma)$ its {\it{$M$-translate}} inside $\calA_g(\Gamma_M)$
is $Z_M:=q_M(p^{-1}(Z))$.
\end{dfx}

Clearly, the $M$-translate of $Z$ can be pulled back to any level that
dominates $\calA_g(\Gamma_M)$.

The following statement gives a way to produce more
absolutely or universally irreducible loci, starting from known ones.

\begin{lm}[Absolute/universal irreducibility is translation invariant]\label{lemma:transl-irr}
Assume $g\geq 2$. Let $Z\subset\calA_g(\Gamma)$ and $M\in\Sp_{2g}(\QQ)$,
and let $Z_M$ be the $M$-translate of $Z$.
\begin{itemize}
\item[(i)]
$Z$ is absolutely irreducible $\iff$ $Z_M$ is absolutely irreducible.
\item[(ii)]
$Z$ is universally irreducible $\iff$ $Z_M$ is universally irreducible.
\end{itemize}
\end{lm}
\begin{proof}
Let $p$ and $q_M$ be as above.
Both claims (i-ii) are straightforward, since
$(\calA_g(\Gamma_M), Z_M)$ is isomorphic to $(\calA_g(\Gamma^M),\,p^{-1}(Z))$ via $q_M$
and $p$ is a finite \'etale cover.
\end{proof}
%
%
%Let $\wti{Z}\subset\HH_g$ be the preimage of $Z'$.
%Since absolute and universal irreducibility are invariant under finite covers,
%we can assume that $M\cdot Z^\dagger$ is defined at a finite level $\Gamma^\dagger\subseteq\Gamma'(M)$.
%
%
%(i) Suppose first $Z'$ absolutely irreducible.
%Let $\Gamma''$ be a finite-index subgroup of $\Gamma^\dagger$.
%By Lemma \ref{lemma:translate}(ii),
%$M\cdot Z''$ is covered by a finite cover $Z'''$ of $Z'$.
%Since $Z'''$ is irreducible, so are $M\cdot Z''$ and $M\cdot Z^\dagger$.
%Hence, $M\cdot Z^\dagger$ is absolutely irreducible.
%
%Suppose now $M\cdot Z^\dagger$ absolutely irreducible.
%Let $\Gamma^\star$ be a finite-index subgroup of $\Gamma'$
%and let $\Gamma''=\Gamma^\star(M)\cap \Gamma^\dagger$, which has finite-index in $\Gamma_g$.
%By Lemma \ref{lemma:translate}(iii) there is a finite cover
%$M\cdot Z''\rar Z^\star$.
%Since $M\cdot Z''$ is irreducible, it follows that $Z^\star$ is irreducible.
%This implies that $Z'$ is absolutely irreducible.
%
%
%(ii) This is immediate, since $\wti{Z}$ is irreducible if and only if
%$M\cdot\wti{Z}$ is.
%\end{proof}

As an example of application of Lemma \ref{lemma:transl-irr}, for $g\geq 3$
all rational translates of 
$\calJ_g(n)$
and of $\tn\chars\ep\de$ (for $n$ even)
are universally irreducible
by Corollary \ref{cor:tn} and by Lemma \ref{lemma:irr-Jacobian}.

%%%%%%%%%%%%%%%%%%%%%%%%%%%%%%%%%%%%%%%%%%%%%%%%%%

\subsection{From absolutely irreducible to universally irreducible}
 
Let $\Gamma$ be a finite-index subgroup of $\Gamma_g$.
The goal of this section is to prove the following.

\begin{thm}[Absolutely irreducible of small codimension are universally irreducible]\label{thm:ag}
Let $g\geq 3$
and let $Z\subset\calA_g(\Gamma)$ be a \CORRECT{complete intersection subvariety} of codimension at most $g-2$. Then its preimage $\wti{Z}$ in $\HH_g$ has finitely many irreducible components.
In particular, if $Z$ is
absolutely irreducible, then it is universally irreducible.
\end{thm}

The proof of Theorem \ref{thm:ag}
relies on the following main ingredients:
\begin{itemize}
\item 
absolute and universal irreducibility are invariant under translation by elements of the rational  symplectic  group
(Lemma \ref{lemma:transl-irr})
\item
each \CORRECT{complete intersection subvariety} $Z\subset\calA_g(\Gamma)$ of codimension at most $g-2$ has a rational translate $Z_M$ that meets the Jacobian locus in a non-hyperelliptic Jacobian (Lemma \ref{lemma:degenerations})
\item
the image of $\pi_1((Z_M)_{\sm})\rar\Gamma_g$
contains a pair of transvections along 
two linearly independent vectors of $\ZZ^{2g}$,
and so does the image of $\pi_1(Z)\rar\Gamma_g$
(Corollary \ref{cor:ample-transvections})
\item
Zariski-dense subgroups of $\Gamma_g$ that contain commuting transvections
associated to a pair of linearly independent vectors of $\ZZ^{2g}$
have finite index (Proposition \ref{prop:commuting}).
\end{itemize}

%In order to proceed toward the proof of Theorem \ref{thm:ag}
We begin by analyzing how subvarieties of small codimension
in $\calA_g$ meet the boundary of the Jacobian locus: this is due to technical
reasons, as it is easier to control the monodromy at infinity
of a subvariety of the moduli space of curves $\calM_g$ 
(as shown in Appendix \ref{sec:dehn})
rather
than of a subvariety of $\calA_g$.

\begin{lm}\label{lemma:degenerations}
Let $g\geq 3$ and $Z\subset\calA_g(\Gamma)$ be a \CORRECT{complete intersection subvariety}
of codimension  $k\leq g-2$ and denote by $p':\calA_g(\Gamma_M)\rar\calA_g$
the natural projection. Then there exists a rational translate
$Z_M\subset\calA_g(\Gamma_M)$ of $Z$ 
such that
\begin{itemize}
\item[(i)] 
the image $p'(Z_M)\subset\calA_g$
meets the non-hyperelliptic
Jacobian locus $\calJ_g\setminus\calH_g$;
\item[(ii)]
if $\calA_{g-h}$ is the largest boundary stratum of $\ol{\calA}_g$
that meets $\ol{p'(Z_M)}\cap\ol{\calJ}_g$, then
$\ol{p'(Z_M)}\cap\ol{\calJ}_g\cap \calA_{g-h}$ is non-compact.
%
%either $\ol{M\cdot Z}$ contains the stratum $\ol{\calA}_{g-2}$ or
%$\ol{M\cdot Z}$ meets the locally-closed stratum $\calA_{g-1}$ in a non-compact subset.
\end{itemize}
\end{lm}

We denote by $\delta^{\IRR}_h$ the locally-closed
locus inside the Deligne-Mumford compactification
$\ol{\calM}_g$ that parametrizes curves
whose normalization has genus $g-h$
(the genus of a smooth disconnected curve being the sum
of the genera of its connected components).

\begin{proof}[Proof of Lemma \ref{lemma:degenerations}] 
\CORRECT{
(i) 
 Let $\wti{Z}$ be the preimage of $Z$ in $\HH_g$. Since $\Sp_{2g}(\RR)$ transitively acts on $\calH_g$, 
 there exists
$M\in\Sp_{2g}(\RR)$ such that $M\cdot\wti{Z}$ does not contain any component of  $\wti{\calH}_g$.
By density we can assume $M$ to be rational.
}

\CORRECT{
Since $Z$ is a complete intersection, so is $Z_M$ and their rational classes are positive multiples of  $\lambda^k$ . Since the pull-back $t_g^*(\lambda^k)$ to $\calM_g$ is not rationally trivial 
(see \cite{faber}), we have that $p'(Z_M)\cap\calJ_g$ is not empty
and its irreducible components have codimension at most $k$ inside $\calJ_g$.
Moreover $p'(Z_M)\cap\calJ_g$ does not entirely contain $\calH_g$. Since every component of $p'(Z_M)\cap\calJ_g$ has dimension at least
$2g-1=\mathrm{dim}(\calH_g)$, it follows that
$p'(Z_M)$ must contain the Jacobian of a smooth non-hyperelliptic curve.
% Let $\wti{Z}$ be the preimage of $Z$ in $\HH_g$.
%Since $\Sp_{2g}(\RR)$ transitively acts on $\HH_g$, there exists
%$N\in\Sp_{2g}(\RR)$ such that $N\cdot\wti{Z}$
%transversely intersects $\wti{\calJ}_g\setminus\wti{\calH}_g$.
%Since $g\geq 3$, the submanifold $\wti{\calJ}_g\setminus\wti{\calH}_g$
%has dimension strictly greater than $g-2$, and so we can find $M\in\Sp_{2g}(\QQ)$ sufficiently close to $N$such that$M\cdot\wti{Z}$ intersects $\wti{\calJ}_g\setminus\wti{\calH}_g$.
% Now  projecting on  $\calA_g(\Gamma^M)$
}
%
%Since $M\cdot \wti{Z}$ meets $\wti{\calJ}_g\setminus \wti{\calH}_g$,
%the corresponding $M\cdot Z\subset\calA_g$ meets $\calJ_g\setminus\calH_g$.

(ii) Let $Y$ be a component of $t_g^{-1}(p'(Z_M))\subset\calM_g$, which thus
has codimension at most $g-2$ in $\calM_g$.
%
%Since $\calJ_g\setminus\calH_g$ is smooth, 
%$Y:=t^{-1}(M\cdot Z)\subset\calM_g$ has codimension at most $g-2$.
Let $\calM^{cpt}_g=\delta^{\IRR}_0$ be the locus inside $\ol{\calM}_g$ consisting of curves of compact type, namely without non-disconnecting nodes, and
recall that holomorphic subvarieties of $\calM^{cpt}_g$ of codimension less than $g$
are non-compact \cite{faber-pandharipande}. 
%
%  Faber, C., Pandharipande, R.: Relative maps and tautological classes, J. Eur. Math. Soc.(JEMS) 7 (2005), 13–49.
%
Thus $\ol{Y}$ must meet the locus $\ol{\delta}^{\IRR}_1$ 
of curves with at least one non-disconnecting node 
and $\ol{Y}\cap\ol{\delta}^{\IRR}_1$ has codimension at most $g-2$ inside $\ol{\delta}^{\IRR}_1$.
Up to choosing a different component of $t_g^{-1}(p'(Z_M))$ as $Y$,
we can assume that $t_g(\ol{Y})$ meets $\calA_{g-h}$
and so
the smallest number of non-disconnecting nodes in a curve
parametrized by $\ol{Y}$ is exactly $h\in[1,g-1]$.

%, and let $\delta^{\IRR}_h$ be locus parametrizing curves in $\ol{\calM}_g$ with $h$ non-disconnecting nodes (and any number of disconnecting ones).
Since $\delta^{\IRR}_h$ has codimension $h-1$ in $\ol{\delta}^{\IRR}_1$,
the intersection $\ol{Y}\cap\delta^{\IRR}_h$ has codimension at most $(g-2)-(h-1)=g-1-h$ inside $\delta^{\IRR}_h$.
Consider now the following commutative diagram
\[
\xymatrix{
\calM^{cpt}_{g-h,2h}\ar[rr]^{\nu} \ar[d]^f && \delta_h^{\IRR}\ar[d]_{\ol{t}_g} \\
\calM^{cpt}_{g-h}\ar[rr]^{t_{g-h}} && \calA_{g-h}
}
\]
in which $f$ sends the $2h$-marked curve $(C,x_1,\dots,x_{2h})$ to $C$
and $\nu$ sends $(C,x_1,\dots,x_{2h})$ to the curve obtained from $C$
by gluing each couple $(x_1,x_2)$, $(x_3,x_4)$, \dots, $(x_{2h-1},x_{2h})$ to a node.

Since $\nu$ is finite surjective and $f$ is a surjective fibration with
$2h$-dimensional fiber, $f(\nu^{-1}(\ol{Y}))$ is non-empty of codimension
at most $g-h-1$ inside $\calM_{g-h}^{cpt}$, and so is non-compact.
Hence $\ol{t}_g(\ol{Y}\cap\delta_h^{\IRR})$ is non-compact inside
$\calA_{g-h}$. The conclusion follows, since $\ol{t}_g(\ol{Y}\cap\delta_h^{\IRR})$
is closed and it is contained inside $\ol{p'(Z_M)}\cap\ol{\calJ}_g\cap\calA_{g-h}$.
\end{proof}

%{\bf  The  conditions of the  lemma are satisfied for   complete intersections, thus} 
The above lemma and the work done in Appendix \ref{sec:dehn} yield the following.

\begin{cor}\label{cor:ample-transvections}
Let $g\geq 3$ and $Z\subset\calA_g(\Gamma)$ be an absolutely
irreducible \CORRECT{complete intersection subvariety}
of codimension at most $g-2$.
Then the image of $\pi_1(Z_{\sm})\rar\Gamma_g$ contains
two commuting transvections associated to a pair of independent vectors.
\end{cor}
\begin{proof}
Let $Z_M\subset\calA_g(\Gamma_M)$ be a rational translate of $Z$
as in Lemma \ref{lemma:degenerations}.
By Lemma \ref{lemma:transl-irr}(i) such $Z_M$ is absolutely irreducible.
By Corollary \ref{cor2:two-twists} applied to the subvariety
$Z_M$ of $\calA_g(\Gamma_M)$,
the image of $\pi_1((Z_{\sm})_M)\rar\Gamma_g$ contains two 
commuting transvections $T_u,T_v$ associated to linearly independent vectors $u,v$.
The isomorphism $Z^M\cong Z_M$ constructed in Section \ref{sec:translates}
shows that the image of $\pi_1(Z^M_{\sm})\rar\Gamma_g$
is obtained from the image of $\pi_1((Z_{\sm})_M)\rar\Gamma_g$
by applying $c_{M^{-1}}$. As a consequence, 
the image of $\pi_1(Z^M_{\sm})\rar\Gamma_g$
contains the commuting transvections
$c_{M^{-1}}(T_u)=T_{M^{-1}u}$ and $c_{M^{-1}}(T_v)=T_{M^{-1}v}$,
associated to the linearly independent vectors $M^{-1}u$ and $M^{-1}v$.
The conclusion
follows, since $Z^M$ covers $Z$ and so the image of $\pi_1(Z_{\sm})\rar\Gamma_g$
contains the image of $\pi_1(Z^M_{\sm})\rar\Gamma_g$.
\end{proof}

Using the previous steps and the arithmeticity criterion of
Section \ref{app:arithm}, we can now prove our main result.

\begin{proof}[Proof of Theorem \ref{thm:ag}]
By Corollary \ref{cor:ample-transvections}
the image of $\pi_1(Z_{\sm})\rar\Gamma_g$ 
contains two linearly independent, commuting transvections.
Since $Z$ is absolutely irreducible, the image of $\pi_1(Z_{\sm})\rar\Gamma_g$ is Zariski-dense
in $\Gamma_g$ by Lemma \ref{lemma:dense}.
As a consequence, 
the image of $\pi_1(Z_{\sm})\rar\Gamma_g$ has finite-index in $\Gamma_g$ by Proposition \ref{prop:commuting}.
The conclusion follows from Corollary \ref{cor:total}.
\end{proof}

The same argument also shows the following.

\begin{prop}\label{prop:genus2}
Assume $g=2$. Let $D\subset\calA_2(\Gamma)$ be an absolutely irreducible, ample
divisor, whose image $D$ in $\ol{\calA}_2$ satisfies $\ol{D}\supset\pa\calA_2$.
Then $D$ is universally irreducible.
\end{prop}

Note that $\tn\chars\ep\de\subset\calA_2(2)$ is not absolutely irreducible,
as discussed in Section \ref{sec:tn2}.\\

We conclude the present section with a few questions.
\begin{itemize}
\item[(1)]
Do there indeed exist absolutely irreducible divisors in $\calA_2(\Gamma)$ with ample closure
(and that contain an irreducible component of $\pa\calA_2(\Gamma)$)?
\item[(2)]
In Proposition \ref{prop:genus2}, is the hypothesis $\ol{D}\supset\pa\calA_2$ necessary?
\item[(3)]
In Theorem \ref{thm:ag}, is the hypothesis $\mathrm{codim}_{\calA_g(\Gamma)}(Z)\leq g-2$ necessary?
\item[(4)]
Under what conditions on $X$, absolutely irreducible subvarieties of $X$
are always universally irreducible? 
\end{itemize}

An interesting case in
the above question (4) is $X=\calM_g$ with $g\geq 2$,
whose universal cover (in the orbifold sense) is 
the Teichm\''uller space, which is isomorphic to
a contractible bounded domain in $\CC^{3g-3}$.

  \section{Examples}\label{sec:examples}
   In this  last section we treat three special examples. 

\subsection{Theta-nulls in genus two}\label{sec:tn2}

Compared to what happens for genus at least $3$, the situation for $\tn\chars\ep\de$ in genus $2$ is completely different.

 \begin{prop}[Theta-nulls are not absolutely irreducible for $g=2$]\label{prop:th2}
  The divisor $\tn\chars\ep\de$ in $\calA_2(2)$ is
  not absolutely irreducible, and $\wti{\tn}\chars\ep\de$ has infinitely many smooth connected components, each isomorphic to $\HH_1\times\HH_1$.
  \end{prop} 

We remark that the hypotheses of Theorem \ref{thm:divisor}
are not satisfied. 
Indeed, the divisor $\tn\subset \calA_2$ is irreducible
and it coincides with the locus $\calA_2^{\mathrm{dec}}$ of decomposable Abelian varieties. Hence, its closure contains $\pa\calA_2$, and so $\pa\tn$ has codimension $1$ inside $\ol{\tn}$. As a consequence,
the same happens for each of the $10$ smooth irreducible components $\tn\chars\ep\de$ of $\tn(2)$ inside $\calA_2(2)$.
% consists of decomposable Abelian variety and $\pa\tn\chars\ep\de$ has codimension $1$ inside $\ol{\tn}\chars\ep\de$. 

\begin{proof}[Proof of Proposition \ref{prop:th2}]
Note that $\pi_1(\tn)$ can be identified to the 
   subgroup $(\Gamma_1\times\Gamma_1)\ltimes \mathfrak{S}_2$
   of $\Sp_4(\ZZ)$, namely
 \[
 \pi_1(\tn)\cong
\left\{
\left(\begin{array}{cc}A & 0\\ 0 & D\end{array}\right),
\ \left(\begin{array}{cc}A\cdot\sigma & 0\\ 0 & D\cdot\sigma\end{array}\right)
\ \Big|\ A,D\in\mathrm{SL}_2(\ZZ)\right\} 
 \]
 where $\sigma=\left(\begin{array}{cc}0 & 1\\ 1 & 0\end{array}\right)
 \in\mathrm{GL}_2(\ZZ)$.
 In view of Lemma \ref{lm:finite-lift}(i), the number of connected components
of the preimage of $\tn\chars\ep\de$ inside $\calA_g(2n)$
is $[\Sp_4(\ZZ/{2n}):\mathrm{SL}_2(\ZZ/{2n})\times\mathrm{SL}_2(\ZZ/{2n})]$
for $n>1$, which increases with $n$.
Hence $\tc\ep\de$ is not absolutely irreducible. 

The last claim is easy, since the reducible locus in $\HH_2$
is obtained acting by $\Sp_4(\ZZ)$ on the locus of diagonal matrices, which can be naturally identified to $\HH_1\times\HH_1$.
 \end{proof}
 
%It follows that $\pi_1(\tn)$ is isomorphic to a subgroup of infinite index in $\Gamma_2$, and so $\wti{\tn} \subset \HH_2$ has infinitely many smooth connected components.
%Note furthermore that the index
%of $\mathrm{SL}_2(\ZZ/{2n})\times\mathrm{SL}_2(\ZZ/{2n})$ 
%inside $\Sp_4(\ZZ/{2n})$ increases with $n$.

%We condense the above considerations into the following.

%  \begin{prop}[Theta-nulls are not absolutely irreducible for $g=2$]\label{prop:th2}
%  The divisor $\tn\chars\ep\de$ in $\calA_2(2)$ is
%  not absolutely irreducible, and $\wti{\tn}\chars\ep\de$ has infinitely many smooth connected components, each isomorphic to $\HH_1\times\HH_1$.
%  \end{prop} 
%\begin{proof}
%In view of Lemma \ref{lm:finite-lift}(i), the number of connected components
%of the preimage of $\tn\chars\ep\de$ inside $\calA_g(2n)$
%is $[\Sp_4(\ZZ/{2n}):\mathrm{SL}_2(\ZZ/{2n})\times\mathrm{SL}_2(\ZZ/{2n})]$
%for $n>1$, which increases with $n$.
%Hence $\tc\ep\de$ is not absolutely irreducible. 
%
%The last claim is easy, since the reducible locus in $\HH_2$
%is obtained acting by $\Sp_4(\ZZ)$ on the locus of diagonal matrices, which can be naturally identified to $\HH_1\times\HH_1$.
%\end{proof}

Proposition \ref{prop:th2} implies that, for $g=2$,
a factorization
of $\tc\ep\de$ into absolutely irreducible forms does not exist (note
that Lemma \ref{lm:factorization} does not apply).

  \subsection{Intermediate Jacobians of cubic threefolds}
  
 We recall that for an odd characteristic
 $\chars\ep\de$ we can define  $\gn\chars\ep\de\subset\calA_g(2)$ as the locus of all ppav's at which the gradient $d_z\tc\ep\de(\tau,0)$ vanishes.
 Such locus has expected codimension $g$ in $\calA_g(2)$.
%As in the even case this has several components in $\calA_g(2)$.

For $g\geq 5$, it is known that each $\gn \chars\ep\de$ has an irreducible component, which we denote by $\calC\chars\ep\de$, that contains some irreducible component of $\calH_g(2)$; 
however, such $\gn \chars\ep\de$ 
has other irreducible components besides $\calC\chars\ep\de$
(for example, inside the decomposable locus $\calA_1(2)\times\calA_{g-1}(2)$
we will find some components of type $\calA_1(2)\times\tn\chars{\ep'}{\de'}$).

In genus $g=5$ we have that, at level $1$, this  component  is the  closure   of the  moduli space 
$\mathcal C$ of  the intermediate Jacobians of cubic threefolds.    In \cite[Section 4]{ACT} it is proven that   $\calH_5$ in $\calC$ is smooth (in an  orbifold sense)  and  the period map 
$$ \pi: \calC \lra \calA_5$$
 extends to a regular map at the generic point of $\calH_5$ in $\calC$.  Hence we have

%In this case the same construction as in Corollary \ref{cor:tn} applies. 
%  and we denote by $\gn$ the image of such locus in $\calA_g$.
 
  \begin{cor}[Irreducibility of $\wti{\calC}\chars\ep\de$]\label{cor:grad} 
  For $g=5$ the locus $\wti{\calC}\chars\ep\de$  is irreducible in $\HH_5$. 
  For $g\geq 6$
  the locus $\calC\chars\ep\de$ lifts to the union $\wti{\calC}\chars\ep\de$ of finitely many irreducible divisors in $\HH_g$.
  \end{cor} 
\begin{proof}  
The proof is an immediate consequence of the  above discussion.
Moreover, Julia Bernatska kindly communicated us that the above-mentioned component of
$\calH_5(2)$ is contained inside the smooth locus of $\calC\chars\ep\de$, as a consequence of  \cite[Remark 9]{ber}.
 \end{proof}

%According to \cite{CL}, in genus 5 each $\calC\chars\ep\de\subset\calA_g(2)$ is a component of the locus of 
%intermediate Jacobians of cubic threefolds with level $2$ structure.
   
   \subsection{The Schottky form}   
    In this section
    we are going to exhibit an absolutely irreducible form of weight greater than $1/2$. 
 This will be the so-called Schottky form.

%We begin by describing it.

  %\subsection{Schottky form}\label{sec:Schottky}

Let $S$ be an integral, positive-definite quadratic form on $\ZZ^m$,
% then  we can
and consider  the theta series
$$\theta_S (\tau)=\sum_{G\in \mathrm{Mat}_{m,g}(\ZZ)}\mathrm{exp}(\pi i\cdot\mathrm{tr}( G^{t}SG\tau  )). $$
%where $G^t$ is the transpose  of $G$.
Such $\theta_S$ is a modular form of weight $m/2$ with respect to a  subgroup  of   finite index $\Gamma \subset  \Gamma_g$.
If $S$ is even and unimodular, then  $\Gamma=  \Gamma_g$. 

In the following lemma we wish to highlight an interesting property of modular forms that are linear combination of theta series.
 
\begin{lm}[Absolute irreducibility of theta series propagates]\label{lemma:propagates}
Let $S_1,\dots,S_k$ be integral, positive-definite quadratic forms on $\ZZ^m$
and let $a_1,\dots,a_k\in\CC$, and consider for all $g\geq 3$
the following modular form
\[
h_g:=\sum_{i=1}^k a_i\theta_{S_i}.
\]
If $h_{g_0}$ is absolutely irreducible, then $h_g$ is universally irreducible for all $g\geq g_0$.
\end{lm}
 \begin{proof}
 We will prove that $h_g$ is an absolutely irreducible form for all $g\geq g_0$.
By Theorem \ref{thm:ag} it will follow that $h_g$ is universally irreducible for all $g\geq g_0$.

In order to show that $h_g$ is an absolutely irreducible form for all $g\geq g_0$,
we proceed by induction.

The initial case $g=g_0$ is given by hypothesis.
In order to deal with the inductive step, we assume $g>g_0$
and we recall that the Siegel operator
 $$\Phi(h)(\tau'): = \lim_{t\to +\infty} h \left(\begin{matrix} \tau ' & *  \\ * &it\end{matrix}\right) $$
sends a modular form $h$ of genus $g$ at level $n$ to a modular form $\Phi(h)$ of genus $g-1$ at level $n$,
in such a way that $\Phi(h_a h_b)=\Phi(h_a)\Phi(h_b)$.

By contradiction, suppose that
the form $h_g$ breaks as a product
  $h_g=h_{g, a}h_{g, b}$
  of modular forms $h_{g,a}$ and $h_{g,b}$ with respect to 
  a finite-index subgroup $\Gamma$ of $\Gamma_g$.
  %Apply then the Siegel  operator, namely consider
   % $$\Phi(f) = \lim_{t\to +\infty} f \left(\begin{matrix} \tau ' & *  \\ * &it\end{matrix}\right) .$$
  % This  operator $\Phi$ sends a modular form of  genus  $g$ to a form of genus $g-1$. In particular,
Since $\Gamma$ contains the full congruence subgroup $\Gamma_g(n)$ for some $n$,
we can assume without loss of generality that $\Gamma=\Gamma_g(n)$.
  Since $h_g$ is a linear combination of theta series we have 
    $$ \Phi(h_g)= h_{g-1}.$$
    This implies that $h_{g-1}=\Phi(h_{g,a})\Phi(h_{g,b})$ is a nontrivial
    factorization of modular forms with respect to $\Gamma_{g-1}(n)$, and so
    $h_{g-1}$ is not irreducible at level $n$. Such contradiction proves that
    $h_g$ is an absolutely irreducible form.
 \end{proof}

We now consider a special case of the above construction.

We recall that, for $m=16$, there are exactly two classes of 
integral, positive-definite quadratic forms on $\ZZ^{16}$: those associated to the lattices $E_8\oplus E_8$ and  $D_{16}^+$ . For any $g$ we consider the form  
  $$f_g=\theta_{E_8\oplus E_8}- \theta_{D_{16}^+}.$$
   It is  a  well-known fact that  $f_g$ vanishes identically for $g=1,2,3$. 
   It is Schottky's form  when $g=4$, cf.~\cite{sch} or  \cite{csb}. 
 For $g\geq 4$ we denote by $F_g$ the  divisor defined  by  $f_g$. We know that $F_4$  in $\calA_4$ is universally irreducible
 by Lemma \ref{lemma:irr-Jacobian}, since it coincides with the closure of the Jacobian locus $\calJ_4$.  
 %Since the preimage of $\calJ_4$ inside
  %$\HH_g$ is irreducible, we have that $f_4$ is a universally prime form.
  %  all covering of  $\calJ_4$ are irreducible, we  have that  $f_4$ is  an absolutely prime form. Then it lifts to an  irreducible divisor in
%  $\HH_4$.
%\smallskip
 % 
  %We want to study the divisor $F_g$ for $g\geq 5$. 
%  We know that in these cases $F_g$ does not contain the  Jacobian locus, cf. \cite{sch} or \cite{csb}. However it contains the hyperelliptic locus $\calH_g$, cf. \cite{poor1}.
   %Thus the preimage $\wti{F}_g \subset \HH_g$ of $F_g$ is connected.
   
   The same Lemma \ref{lemma:irr-Jacobian} cannot be applied in higher genus: though $F_g$ contains the hyperelliptic locus $\calH_g$
   (see \cite{poor1}), for $g\geq 5$ it does not contain the  Jacobian locus (see \cite{sch} or  \cite{csb}).
    
   Nevertheless, an immediate application of Lemma \ref{lemma:propagates} gives
  \begin{cor}[The Schottky form is universally irreducible]\label{cor:sch}
  The divisor $F_g$ of $\calA_g$ is universally irreducible for all $g\geq 4$.
  \end{cor}

\section{A criterion for arithmeticity in $\Sp_{2g}(\ZZ)$}\label{app:arithm}
 
%Let $\Omega$ be a non-degenerate symplectic form on $\QQ^{2g}$,
%which is integral on the standard lattice $\ZZ^{2g}$. 
%
%\begin{example}[Standard symplectic form]
We recall that the standard symplectic form $\omega$ on $\QQ^{2g}$, defined
by $\omega(e_i,e_j)=-\omega(e_j,e_i):=\delta_{j,g+i}$ for all $1\leq i\leq j\leq 2g$, is nondegenerate and integral on $\ZZ^{2g}$.

If $W\subseteq\QQ^{2g}$ is a vector subspace, we
denote by $\Sp(W,\omega)$ the group of linear automorphisms of $W$ that preserve the restriction of $\omega$ to $W$. We 
remark that $\Sp(\ZZ^{2g},\omega)=\Sp_{2g}(\ZZ)$ 
and that $\Sp(\QQ^{2g},\omega)=\Sp_{2g}(\QQ)$.

\begin{dfx}[Transvections in $(\ZZ^{2g},\omega)$]\label{def:transvections}
The {\it{transvection}} associated to $w\in\ZZ^{2g}$
is the integral symplectic transformation $T_w:\ZZ^{2g}\rar\ZZ^{2g}$ defined as $T_w(v):=v+\omega(v,w)w$. Two transvections $T_u,T_v$ are {\it{linearly independent}}
if $u,v$ are linearly independent as vectors of $\QQ^{2g}$.
\end{dfx}

Observe that $(T_w-I)(\ZZ^{2g})=\ZZ\cdot aw$ for some integer $a\neq 0$.

The aim of this short section is to prove the following.

\begin{prop}[A criterion of arithmeticity]\label{prop:commuting}
Let $G$ be a Zariski-dense subgroup of $\Sp_{2g}(\ZZ)$ that contains two
linearly independent, commuting transvections
$T_u,T_v$.
Then $G$ has finite index in $\Sp_{2g}(\ZZ)$.
\end{prop}

The proof of the above proposition will make essential use of

%
%If $W\subset\QQ^n$ is a vector subspace, $W_{\ZZ}=W\cap\ZZ^n$ and $w\in\ZZ^n\cap W$,
%then the restriction of the transvection $T_w$ to the subspace $W$
%is the integral symplectic transformation
%$T'_w:W_\ZZ\rar W_\ZZ$ defined as $T'_w(v):=v+\Omega_W(v,w)w$, where $\Omega_W$ is the restriction
%of $\Omega$ to $W$.

\begin{thm}[Singh-Venkataramana {\cite[Theorem 1.2]{SV}}]\label{thm:SV}
%Let $\Omega$ be a non-degenerate symplectic form on $\QQ^{2g}$, which is integral on $\ZZ^{2g}$, and s
Suppose that $G\subseteq \Sp_{2g}(\ZZ)$ is a subgroup that satisfies the following properties:
\begin{itemize}
\item[(i)]
$G$ is Zariski dense inside $\Sp_{2g}(\ZZ)$;
\item[(ii)]
there exist three vectors $w_1,w_2,w_3\in\ZZ^{2g}$ such that
\begin{itemize}
\item
$W = \sum_{i=1}^3 \QQ\cdot w_i$ of $\QQ^{2g}$
is $3$-dimensional;
\item
$\Omega(w_i,w_j) \neq 0$ for some $i,j$;
\item
the group $G$ contains the transvections $T_{w_1},T_{w_2},T_{w_3}$;
\item
the subgroup of $\Sp(W,\omega)$
generated by the restrictions $T'_{w_1},T'_{w_2},T'_{w_3}$ to $W$ of the transvections $T_{w_1},T_{w_2},T_{w_3}$
contains a non-trivial element of the unipotent radical of $\Sp(W,\omega)$.
\end{itemize}
\end{itemize}
Then the group $G$ has finite index in
$\Sp_{2g}(\ZZ)$.
\end{thm}

\begin{rem}
The above result holds if $\omega$ is replaced by any non-degenerate
symplectic form on $\QQ^{2g}$, which is integral on $\ZZ^{2g}$.
\end{rem}

We wish to apply Theorem \ref{thm:SV} in the situation
described in the following lemma.

\begin{lm}[Finding elements in the unipotent radical]\label{lemma:radical}
Let $w_1,w_2,w_3\in\ZZ^{2g}$ be linearly independent
and such that $\omega(w_1,w_2),\omega(w_1,w_3)\neq 0$ and $\omega(w_2,w_3)=0$.
Then the group generated by the restrictions of $T_{w_1},T_{w_2},T_{w_3}$ to
$W=\mathrm{Span}(w_1,w_2,w_3)$ contains a nontrivial element of the unipotent radical of $\Sp(W,\omega)$.
\end{lm}
\begin{proof}
Note that it is enough to prove the statement with $(w_1,w_2,w_3)$ replaced by nonzero multiples $(k_1w_1,k_2w_2,k_3w_3)$.

Let $e\in\ZZ^n$ be a $\QQ$-generator of $\mathrm{Rad}(W,\omega)=\QQ\cdot e$,
and write $w_3=a_1 w_1+a_2 w_2+be$, with $a_1,a_2,b\in\QQ$ and $b\neq 0$.
Up to replacing $w_3$ with a multiple, we can assume that $a_1,a_2,b\in\ZZ$.

Let $\mathcal{B}=(w_1,w_2,e)$ be a $\QQ$-basis of $W$.
Then elements of the unipotent radical of $\Sp(W,\omega)$ are represented by the matrices
\[
\left(\begin{array}{ccc}
1 & 0 & 0\\
0 & 1 & 0\\
%\hline
s & t & 1
\end{array}\right),\qquad\text{with $t,s\in\QQ$}
\]
with respect to the basis $(w_1,w_2,e)$ of $W$.
Call now $\omega_{ij}=\omega(w_i,w_j)$
and note that $\omega_{13}=a_2\omega_{12}\neq 0$ and $\omega_{23}=a_1\omega_{21}=0=a_1$.
Then the restrictions $T'_{w_1}$ of $T_{w_i}$ to $W$ can be written as
\begin{align*}
& T'_{w_1}=\left(\begin{array}{ccc}
1 & \omega_{21} & 0\\
0 & 1 & 0\\
%\hline
0 & 0 & 1
\end{array}\right),\qquad
T'_{w_2}=\left(\begin{array}{ccc}
1 & 0 & 0\\
\omega_{12} & 1 & 0\\
%\hline
0 & 0 & 1
\end{array}\right),\\
& T'_{w_3}=\left(\begin{array}{ccc}
1 & 0 & 0\\
a_2^2\omega_{12} & 1 & 0\\
%\hline
ba_2\omega_{12} & 0 & 1
\end{array}\right)
\end{align*}
with respect to the basis $\mathcal{B}$.
It follows that
\[
(T'_{w_2})^{-a_2^2}T'_{w_3}=
\left(\begin{array}{ccc}
1 & 0 & 0\\
-a_2^2\omega_{12} & 1 & 0\\
%\hline
0 & 0 & 1
\end{array}\right)\left(\begin{array}{ccc}
1 & 0 & 0\\
a_2^2\omega_{12} &1 & 0\\
%\hline
ba_2\omega_{12} & 0 & 1
\end{array}\right)=
\left(\begin{array}{ccc}
1 & 0 & 0\\
0 &1 & 0\\
%\hline
ba_2\omega_{12} & 0 & 1
\end{array}\right).
\]
Since $ba_2\neq 0$, the element $(T'_{w_2})^{-a_2^2}T'_{w_3}\neq I$ is in the unipotent radical of $\Sp(W,\omega)$.
\end{proof}

As a consequence, we easily obtain our wished criterion.

\begin{proof}[Proof of Proposition \ref{prop:commuting}]
The subset of $\Sp_{2g}(\CC)$ consisting of elements
$\tau\in\Sp_{2g}(\CC)$ such that $\omega(\tau(u),u)\omega(\tau(u),v)=0$
is a proper algebraic subvariety of $\Sp_{2g}(\CC)$.
Since $G$ is Zariski-dense, there exists $\gamma\in G$ such that
$\omega(\gamma(u),u)\neq 0$ and $\omega(\gamma(u),v)\neq 0$.

Now, $u,v$ are linearly independent, and $T_u,T_v$ commute, which implies that $\omega(u,v)=0$.
Hence, the vectors $w_1=u$, $w_2=\gamma(u)$, $w_3=v$ are linearly independent
and $\omega(w_1,w_2),\omega(w_2,w_3)\neq 0$ whereas $\omega(w_1,w_3)=0$.
Moreover, $G$ contains $T_{w_1}$, $T_{w_2}=\gamma T_{w_1}\gamma^{-1}$, $T_{w_3}$.
Lemma \ref{lemma:radical} then implies that the hypotheses of Theorem \ref{thm:SV} are satisfied, and so
$G$ has finite index in $\Sp_{2g}(\ZZ)$.
\end{proof}

In order to verify the hypothesis of Zariski-density in Proposition \ref{prop:commuting}
when dealing with absolutely irreducible subvarieties in a finite \'etale cover of $\calA_g$,
we will make use of the following observation.

\begin{lm}\label{lemma:dense}
Let $\Gamma$ be a finite-index subgroup of $\Gamma$ and let $Z\subset\calA_g(\Gamma)$
be an absolutely irreducible subvariety.
Then the image of $\pi_1(Z_{\sm})\rar\Gamma_g$
is Zariski dense.
\end{lm}
\begin{proof}
Absolute irreducibility of $Z$ implies that
the image of $\pi_1(Z_{\sm})\rar\Gamma$
surjects onto every finite quotient of $\Gamma$.
Since $\Gamma_g$ is residually finite,
$G$ is Zariski-dense in $\Gamma$
by \cite[Proposition 3]{margulis}, and so is Zariski-dense in $\Gamma_g$.
\end{proof}

\appendix

\renewcommand{\thesection}{\Roman{section}}

\section{Subvarieties of $\calM_g$, Dehn twists and transvections}\label{sec:dehn}

%For technical reasons related to a  understanding of the  monodromy, we will analize the situation of moduli of curves.

We recall that the mapping class group $\MCG_g$ can be identified
to the orbifold fundamental group of $\calM_g$. Moreover the homomorphism
induced by the Torelli morphism $t_g:\calM_g\rar\calA_g$
at the level of fundamental groups $\pi_1(t_g):\pi_1(\calM_g)\rar\pi_1(\calA_g)$
identifies to the standard symplectic representation
$\rho:\mathrm{MCG}_g\rar \Gamma_g=\Sp_{2g}(\ZZ)$.
We recall the following classical result (see, for instance, \cite[Theorem 6.4]{farb-margalit}).

\begin{fact}\label{fact:rho-surjective}
$\rho$ is surjective.
\end{fact}
%Such $\rho$ is classically known to be surjective
%(see, for instance, Chapter 6 of \cite{farb-margalit}).\\

The following is an immediate consequence.

\begin{cor}\label{cor:rho-surjective}
For every subgroup $\Gamma$ of $\Gamma_g$,
the natural homomorphism $\pi_1(\calJ_g(\Gamma))\rar\Gamma$ is surjective, where $\calJ_g(\Gamma)$ is
the Jacobian locus in $\calA_g(\Gamma)$.
\end{cor}

We also recall that the Torelli morphism 
$t_g$ can be
extended to a map 
$\ol{t}_g:\ol{\calM}_g\rar\ol{\calA}_g$
from the Deligne-Mumford compactification $\ol{\calM}_g$ to the Satake compactification $\ol{\calA}_g$.

In order to state Proposition \ref{prop:two-twists},
we denote by $\delta^{\IRR}_h$ the locally-closed locus inside $\ol{\calM}_g$
that parametrizes curves whose normalization has genus $g-h$
(the genus of a disconnected curve being the sum of the genera of its connected components): in other words, points in $\delta^{\IRR}_h$ are isomorphism classes of
stable curves with exactly $h$ non-disconnecting nodes (and possibly other disconnecting nodes).
We observe that $\delta^{\IRR}_h$ is a subvariety of codimension $h$.

%lifts to an injective map
%$t:\calM\rar\calA$, where $\calM$ is the cover of $\calM_g$
%associated to the subgroup $\rho^{-1}(\Gamma)\subseteq\mathrm{MCG}_g$.
%Such $t$ extends to a map $\ol{t}:\ol{\calM}\rar\ol{\calA}^{Sat}$
%from the Deligne-Mumford compactification
%$\ol{\calM}$ to the Satake compactification $\ol{\calA} $.
%We also denote by $\calM'$ a finite \'etale cover of $\calM$
%such that $\ol{\calM}'$ is a smooth manifold
%(the existence of such covers is established in \cite{looijenga:prym}).

In this appendix we want to show the following result
(recall Definition \ref{def:transvections}).

\begin{prop}[Finding a pair of commuting transvections I]\label{prop:two-twists}
Let $Y$ be an irreducible subvariety in $\calM_g$.
Suppose that
\begin{itemize}
\item[(a)]
$\ol{Y}$ intersects the locally-closed boundary stratum $\delta^{\IRR}_h$ for some $h>0$;
\item[(b)]
the intersection $\ol{Y}\cap\delta^{\IRR}_h$ is non-compact.
\end{itemize}
Then the image of $\pi_1(Y_{\sm})\rar \Gamma_g$ contains
two linearly independent, commuting transvections.
\end{prop}

In order to prove Proposition \ref{prop:two-twists},
we first recall that $\ol{\calM}_g$ supports a universal family of stable curves with
total space $\ol{\calC}_g$.
Moreover, a path between $b_0\in \calM_g$ and $b\in \pa\calM_g$
induces a map $C_{b_0}\rar C_b$ that shrinks disjoint loops of $C_{b_0}$ to nodes of $C_b$: call them {\it{shrinking loops}}.

\begin{notation}
If $\gamma\subset C_{b_0}$ is a non-trivial simple loop, then
we denote by $\mathrm{Tw}_\gamma\in\MCG(C_{b_0})$ the {\it{right Dehn twist}} along $\gamma$ (see \cite[Chapter 3]{farb-margalit}).
If $\gamma_1,\dots,\gamma_k$ are non-trivial disjoint simple loops on $C_{b_0}$ and $a_1,\dots,a_k>0$ are integers, then
$\mathrm{Tw}_{\gamma_1}^{a_1}\cdots \mathrm{Tw}_{\gamma_k}^{a_k}$ is the element of $\MCG(C_{b_0})$ obtained by 
performing $a_i$ Dehn twists along $\gamma_i$ for $i=1,\dots,k$.
\end{notation}

The following statement is rather standard.

\begin{lm}[Boundary strata and Dehn twists]\label{lemma:twist}
Let $\ol{Y}$ be an irreducible (not necessarily closed) analytic subvariety of $\ol{\calM}_g$, call $Y:=\ol{Y}\cap\calM_g$ and let $b\in\pa Y$.
Suppose $b$ belongs to a locally-closed boundary stratum $\delta$ 
of $\ol{\calM}_g$.
Fix $b_0\in Y_{\sm}$ and a path in $Y_{\sm}\cup\{b\}$ from $b_0$ to $b$
and let $\gamma_1,\dots,\gamma_k\subset C_{b_0}$ be the induced shrinking loops.
Then there exists $a_1,\dots,a_k\geq 1$ such that
the image of $\pi_1(Y_{\sm},b_0)\rar \mathrm{MCG}(C_{b_0})$ contains 
$\mathrm{Tw}_{\gamma_1}^{a_1}\cdots \mathrm{Tw}_{\gamma_k}^{a_k}$.
Hence, the image of $\pi_1(Y_{\sm},b_0)\rar \Sp(H_1(C_{b_0}))$
contains the transvection associated to the vector
$[a_1\gamma_1+\dots+a_k\gamma_k]$.
\end{lm}
\begin{proof}
Pick a holomorphic map $f:\Delta\rar \ol{Y}$ such that $f(0)=b$ and
$f(\Delta^*)$ is contained in $Y_{\sm}$.
Since $Y_{\sm}$ is connected, we can choose $b_0$ as $b_0=f(\tilde{b}_0)$ for some $0\neq\tilde{b}_0\in\Delta$.
Every branch of $\pa\calM_g$ that contains $\delta$ corresponds to a loop $\gamma_i$ on $C_{b_0}$.
If $f$ has multiplicity $a_i\geq 1$ along the $i$-th branch at $\tilde{b}_0$, the image of $\pi_1(\Delta^*,\tilde{b}_0)\rar\mathrm{MCG}(C_{b_0})$
is generated by $\prod_{i=1}^k\mathrm{Tw}_{\gamma_i}^{a_i}$, and the conclusion follows.
\end{proof}

The claimed result is an easy consequence of the above lemma.

\begin{proof}[Proof of Proposition \ref{prop:two-twists}]
Let $b_2$ be a point in the boundary of $\ol{Y}\cap\delta^{\IRR}_h$.
Pick a contractible neighbourhood $U$ of $b_2$ inside $\ol{\calM}_g$
such that $U\cap\ol{\delta}^{\IRR}_h$ and $U\cap \ol{Y}$ are contractible too.
Let $b_0\in U\cap Y_{\sm}$ and $b_1\in U\cap \delta^{\IRR}_h$.
Connect $b_0$ to $b_1$ and to $b_2$ through paths contained in $Y_{\sm}$
(except at $b_1,b_2$). Call $\beta_1,\dots,\beta_k$ the disjoint loops
in $C_{b_0}$ that are shrunk by the map $C_{b_0}\rar C_{b_1}$.
Up to isotopy, we can assume that the map $C_{b_0}\rar C_{b_2}$ shrinks
$\beta_1,\dots,\beta_k$ and the other loops $\beta_{k+1},\dots,\beta_m$ to nodes.
By Lemma \ref{lemma:twist} the image of $\pi_1(Y_{\sm})\rar\Gamma_g$ contains
the commuting transvections $T_v,T_w$, where $v=a_1\beta_1+\cdots+a_k\beta_k$
and $w=c_1\beta_1+\dots+c_m\beta_m$ with $a_i,c_i\geq 1$.
Since $h\geq 1$, at least one loop $\beta_1,\dots,\beta_k$ is non-disconnecting and so the vector $v$ is non-zero.
Since $b_2\in\ol{\delta}^{\IRR}_{h+1}$, 
there exists a loop $\beta_i$ with $k+1\leq i\leq m$ such that
$C_{b_0}\setminus(\beta_1\cup\dots\cup\beta_k)$
and $C_{b_0}\setminus(\beta_1\cup\dots\cup\beta_k\cup\beta_i)$
have the same number of connected components.
It follows that $w$ is not a multiple of $v$, and so $v,w$ are linearly independent.
\end{proof}

We now want to use Proposition \ref{prop:two-twists}
to obtain a similar statement but for subvarieties of $\calA_g(\Gamma)$
instead of $\calM_g$, where $\Gamma$ is any finite-index subgroup of $\Gamma_g$.

We begin with an elementary lemma.

\begin{lm}\label{lemma:commensurable}
Let $\iota:Z\rar\calA_g$ the inclusion of a subvariety and
$N,M$ be subgroups of $\pi_1(Z)$ such that $[N:M\cap N]<+\infty$.
Denote by $\iota_*:\pi_1(Z)\rar \pi_1(\calA_g)=\Gamma_g$ the homomorphism induced by $\iota$.
If $\iota_*(N)$ contains a pair of linearly independent, commuting transvections, then $\iota_*(M)$ does.
\end{lm}
\begin{proof}
Let $\alpha,\beta\in N$ such that
$T_v=\iota_*(\alpha),T_w=\iota_*(\beta)\in \iota_*(N)$ are two linearly independent, commuting transvections.
There exists an integer $m\geq 1$
such that $\alpha^m,\beta^m\in M\cap N$.
It follows that $T_{mv}=\iota_*(\alpha^m),T_{mw}=\iota_*(\beta^m)$ is the wished pair of linearly independent, commuting
transvections in $\iota_*(M)$.
\end{proof}

In the below lemma we prove the wished statement
for subvarieties of $\calA_g$.

\begin{lm}[Finding a pair of commuting transvections II]\label{lemma:two-twists}
Let $Z\subset\calA_g$ be an irreducible subvariety that meets the Jacobian locus,
and suppose that $\ol{Z}\cap\ol{\calJ}_g$
meets the locally-closed boundary stratum $\calA_{g-h}$ in a non-compact subset.
Then the image of $\pi_1(Z_{\sm})\rar  \Gamma_g$ contains
two linearly independent, commuting transvections.
\end{lm}
\begin{proof}
%Suppose first that $Z$ contains the Jacobian locus $\calJ_g$.
%As in Proposition \ref{prop:irr-Mg}(i) and Corollary \ref{cor:criterion}(i),
%the image of $\pi_1(Z_{\sm})\rar  \Gamma_g$ has finite index
%and so the conclusion obviously holds.
%
%Suppose now that $Z$ does not contain $\calJ_g$ and so $t^{-1}(Z)$ is a proper subset of $\calM_g$.
Since $\ol{Z}\cap\ol{\calJ}_g$ inside $\ol{\calA}_g$ meets the boundary stratum $\calA_{g-h}$ in a non-compact
subset,
the closure of $t^{-1}(Z)$ inside $\ol{\calM}_g$ meets 
the locally-closed boundary stratum $\delta^{\IRR}_h$ in a non-compact subset.
Thus, an irreducible component $Y$ of $t^{-1}(Z)$ satisfies hypotheses
(a) and (b) in Proposition \ref{prop:two-twists},
and so the image of $\pi_1(Y_{\sm})\rar \Gamma_g$ contains two linearly independent, commuting transvections.

Call $H$ the image of $\pi_1(Y_{\sm})\rar\pi_1(Z)$ and $K$ the image of $\pi_1(Z_{\sm})\rar\pi_1(Z)$.
Since $Y\rar Z$ is a finite map and $Y,Z$ are irreducible,
Corollary \ref{cor:image-smooth}(i) implies that $[H:K\cap H]<+\infty$.
By Lemma \ref{lemma:commensurable} applied to $N=H$ and $M=K$,
the image of $\pi_1(Z_{\sm})\rar  \Gamma_g$
contains two linearly independent, commuting transvections.
\end{proof}

As a consequence, we obtain our criterion for subvarieties
of $\calA_g(\Gamma)$.

\begin{cor}[Finding a pair of commuting transvections III]\label{cor2:two-twists}
Let $\Gamma$ be a finite-index subgroup of $\Gamma_g$,
let $p:\calA_g(\Gamma)\rar\calA_g$ be the natural projection
and let $Z\subset\calA_g(\Gamma)$ be an irreducible subvariety.
Suppose that the image $p(Z)$ meets the Jacobian locus inside $\calA_g$, and that
$\ol{p(Z)}\cap\ol{\calJ}_g$ intersects the boundary stratum
$\calA_{g-h}$ in a non-compact subset.
Then the image of $\pi_1(p(Z)_{\sm})\rar \Gamma_g$ contains
two linearly independent, commuting transvections.
\end{cor}

\begin{proof}
%Let $Z$ be the image of $Z'$ in $\calA_g$, which is irreducible.
Observe that $p(Z)$ is irreducible, and
let $\iota:p(Z)\rar\calA_g$ be the inclusion and $\iota_*:\pi_1(p(Z))\rar\Gamma_g$ the induced homomorphism.
Moreover, denote by $H$ the image of $\pi_1(Z_{\sm})\rar\pi_1(p(Z))$ and
by $K$ the image of $\pi_1(p(Z)_{\sm})\rar\pi_1(p(Z))$.

By Lemma \ref{lemma:two-twists}, 
the subgroup $\iota_*(K)$
contains two linearly independent, commuting transvections.
Since $Z\rar p(Z)$ has finite fibers,
$H$ is a finite-index subgroup of $K$
by Corollary \ref{cor:image-smooth}(ii) in the case $Y=Z$.
Hence, $\iota_*(H)$ is a finite-index subgroup of $\iota_*(K)$,
and so it contains two linearly independent, commuting transvections
by Lemma \ref{lemma:commensurable} applied to $N=K$ and $M=H$.
\end{proof}

%%%%%%%%%%%%%%%%%%%%%%%%%%%%%%%%%%%%%%%%%%%%%%%%%%%%%

%\input{Rivestimento-S-new}

%\input{Rivestimento-fg-new}

\bibliographystyle{amsalpha}
%\bibliography{../sam_biblio}

\end{document}